\documentclass[11pt]{amsart} \textwidth=14.5cm \oddsidemargin=1cm
\usepackage{amsmath,wasysym}
\usepackage{amsxtra}
\usepackage{amscd}
\usepackage{amsthm}
\usepackage{amsfonts}
\usepackage{amssymb}
\usepackage{eucal}
\usepackage{graphics, color}
\usepackage{hyperref,mathrsfs}
\usepackage[usenames,dvipsnames]{xcolor}
\usepackage{tikz}
\usepackage{stmaryrd}
\usepackage[all]{xy}
\usepackage{hyperref}
\usepackage{color}
\usepackage{bbm} 
\allowdisplaybreaks

\hypersetup{colorlinks,linkcolor=blue, urlcolor=blue,
citecolor=blue}

\newtheorem{thm}{Theorem} [section]
\newtheorem{lem}[thm]{Lemma}

\newtheorem{prop}[thm]{Proposition}

\theoremstyle{definition}

\newtheorem{example}[thm]{Example}

\newtheorem{rem}[thm]{Remark}

\numberwithin{equation}{section}

\newcommand{\A}{\mathcal A}
\newcommand{\bu}{\mathbf i^{-}}
\newcommand{\al}{\alpha}
\newcommand{\B}{{\bf B}}
\newcommand{\sB}{{\mathscr B}}
\newcommand{\C}{{\mathbb C}}
\newcommand{\D}{\mathcal{D}}
\newcommand{\del}{\delta}
\newcommand{\vep}{\delta}
\newcommand{\End}{{\mathrm{End}}}
\newcommand{\F}{\mathbb{F}}
\newcommand{\ff}{\mathtt{f}}
\newcommand{\sF}{{\mathscr F}}
\newcommand{\g}{{\mathfrak g}}
\newcommand{\h}{{\mathfrak h}}
\newcommand{\hf}{{1 \over 2}}
\newcommand{\HH}{\mathcal H}
\newcommand{\Hom}{{\mathrm{Hom}}}
\newcommand{\id}{\mathbbm{1}}
\newcommand{\Ind}{{\mathrm{Ind}}}
\newcommand{\la}{\lambda}
\newcommand{\mi}{\mathbf{i}}
\newcommand{\mj}{\mathbf{j}}
\newcommand{\mk}{\mathbf{k}}
\newcommand{\ml}{\mathbf{l}}
\newcommand{\N}{{\mathbb N}}
\newcommand{\Ob}{\mathscr O}
\newcommand{\OO}{\mathcal O}
\newcommand{\Pro}{\mathrm{Pr}}
\newcommand{\Q}{\mathbb {Q}}

\newcommand{\q}{\mathbf{q}}
\newcommand{\Sc}{\mathcal S}
\newcommand{\Scg}{{\mathcal S}_q'} 
\newcommand{\T}{\mathbb T}

\newcommand{\Z}{{\mathbb Z}}

\newcommand{\nc}{\newcommand}
\nc{\browntext}[1]{\textcolor{brown}{#1}}
\nc{\greentext}[1]{\textcolor{green}{#1}}
\nc{\redtext}[1]{\textcolor{red}{#1}}
\nc{\bluetext}[1]{\textcolor{blue}{#1}}
\nc{\brown}[1]{\browntext{ #1}}
\nc{\green}[1]{\greentext{ #1}}
\nc{\red}[1]{\redtext{ #1}}
\nc{\blue}[1]{\bluetext{ #1}}

\title[The $q$-Schur algebras and  dualities of finite type]
{The $q$-Schur algebras and $q$-Schur dualities \\ of finite type}

\author[Li Luo]{Li Luo}
\address{School of mathematical Sciences,
Shanghai Key Laboratory of Pure Mathematics and Mathematical Practice,
    East China Normal University, Shanghai 200241, China}
\email{lluo@math.ecnu.edu.cn (Luo)}

\author[Weiqiang Wang]{Weiqiang Wang}
\address{ Department of Mathematics\\ University of Virginia\\ Charlottesville, VA 22904}
\email{ww9c@virginia.edu (Wang)}


\begin{document}

\begin{abstract}
We formulate a $q$-Schur algebra associated to an arbitrary $W$-invariant finite set $X_{\texttt f}$ of integral weights for a complex simple Lie algebra with Weyl group $W$. We establish a $q$-Schur duality between the $q$-Schur algebra and Hecke algebra associated to $W$. We then realize geometrically the $q$-Schur algebra and duality, and construct a canonical basis for the $q$-Schur algebra with positivity. With suitable choices of $X_{\texttt f}$ in classical types, we recover the $q$-Schur algebras in the literature. Our $q$-Schur algebras are closely related to the category $\mathcal O$, where the type $G_2$ is studied in detail.
\end{abstract}

\maketitle
\setcounter{tocdepth}{1}
\tableofcontents

\section{Introduction}

\subsection{}

The $q$-Schur algebra (of type A) admits a multiple of formulations.
An algebraic (or DJ) definition in terms of permutation modules of the Hecke algebra was given by Dipper and James \cite{DJ89}, and a geometric (or BLM) definition in terms of $n$-step flags was given by Beilinson, Lusztig and MacPherson \cite{BLM90}. The $q$-Schur duality is a double centralizer property between a $q$-Schur algebra and a Hecke algebra; there is also a version of the $q$-Schur duality due to Jimbo \cite{Jim86} where the $q$-Schur algebra is replaced by a quantum group of type A. A geometric realization of the $q$-Schur duality was given in \cite{GL92}. The $q$-Schur algebra admits a canonical basis in the geometric setting \cite{BLM90}, and an algebraic construction of the canonical basis was given in \cite{Du92}. We refer to the book \cite{DDPW} for a comprehensive account of the (type A) algebraic and geometric constructions.

The classical counterpart of the $q$-Schur algebras is known as Schur algebras, and they arise in the celebrated Schur $(GL(n), S_d)$-duality or $(\mathfrak{gl}(n), S_d)$-duality, between the general linear group $GL(n)$ (or the general linear Lie algebra $\mathfrak{gl}(n)$) and the symmetric group $S_d$.

There has been various generalizations of Schur algebras and $q$-Schur algebras in the literature. In a series of papers starting in \cite{Don86} Donkin formulated a family of generalized Schur algebras associated to a reductive group $G$ of arbitrary type (in place of $GL(n)$ above) and a weight interval of $G$. These algebras are quasi-hereditary algebras and play a basic role in representation theory of algebraic groups. The $q$-deformations of these generalized Schur algebras and their presentations were formulated by Doty \cite{Dot03}; being a quotient of a modified quantum group by a based ideal, these $q$-Schur algebras admit canonical bases (inherited from Lusztig's canonical basis of the modified quantum group). Some of these $q$-Schur algebras admit geometric realization \cite{Li10, DL14} via variants of quiver varieties which are more sophisticated than flag varieties.

As a generalization of \cite{DJ89}, a version of $q$-Schur algebras as centralizer algebras of certain modules over Hecke algebras of type B was formulated independently by Dipper-James-Mathas and Du-Scott \cite{DJM98a, DS00} (also cf. \cite{DJM98b} for a further cyclotomic generalization).

\subsection{}

The goal of this paper is to formulate and study a new family of $q$-Schur algebras of arbitrary finite type (over a ring $\A=\Z[q,q^{-1}]$ or a field $\Q(q)$). In contrast to the settings of the aforementioned works of Donkin and Doty, the type here refers to the type of Weyl groups $W$ (or the associated Hecke algebras) which replace $S_d$ in the classical Schur duality. Another data underlying our constructions is a $W$-invariant weight set for a Lie algebra $\mathfrak g$ whose Weyl group is $W$. (Beware that $\mathfrak g$ has nothing to do with the group $G$ in Donkin's setting. For example, in the standard Schur duality example, $\mathfrak g= \mathfrak{gl}(d)$ while $G=GL(n)$.) Our constructions are motivated by and have applications to the BGG category $\mathcal O$ of $\mathfrak g$-modules.

A type B generalization of the $q$-Schur algebras \`a la Dipper-James, different from \cite{DJM98a, DS00}, was given earlier by R.~Green \cite{Gr97} (who called it a hyperoctahedral Schur algebra).
Our interest in the $q$-Schur algebras of type B \`a la R.~Green and beyond was stimulated by its relevance to the category $\OO$ and a Jimbo type duality involving Hecke algebra of type B \cite{BW18}. A geometric realization \`a la BLM was subsequently given in \cite{BKLW18} using the type B/C flags (also cf. \cite[Appendix A]{FLLLW}), where canonical bases were constructed. We refer to \cite{FL15, Bao17} for the type D generalizations. All these constructions (on $q$-Schur algebra level) are special cases of our constructions in this paper in classical types.

In this paper we develop both algebraic (\`a la DJ) and geometric (\`a la BLM) approaches to such $q$-Schur algebras of arbitrary finite type and their canonical bases; we develop both approaches to a $q$-Schur duality between a $q$-Schur algebra and a Hecke algebra.
We also establish connections of $q$-Schur algebras to the BGG category $\OO$, treating the type $G_2$ in detail. Since we deal with arbitrary finite (including exceptional) types in this paper, our constructions are restricted to the $q$-Schur algebra level and cannot be formulated on the quantum group level.

Informally speaking, Donkin-Doty ($q$-)Schur algebras and our $q$-Schur algebras are generalizations of the type $A$ Schur duality from the two distinct  ``Schur dual'' sides, respectively.  Already in type $A$, our generalization does not coincide with theirs; see Example~\ref{ex:S=H}.
Our $q$-Schur algebra always affords naturally a Schur duality with a Hecke algebra, which were not available in general in the Donkin-Doty setting. To the best knowledge of the authors, the $q$-Schur algebras in this paper and the $q$-Schur algebras of Donkin-Doty are not related beyond type $A$. Our $q$-Schur algebras of type B/C also differ from those in \cite{DJM98a, DS00}; the $q$-Schur algebras {\em loc. cit.} are quasi-hereditary while ours are not in general.

\subsection{}
Let us explain our constructions in some detail.
Let $\g$ be an arbitrary complex simple Lie algebra with $W$ as its Weyl group. Our definition of the $q$-Schur algebra $\Sc_q =\Sc_q(X_\ff)$ relies on a choice of a finite set $X_\ff$ of integral weights for $\g$ which is invariant under the action of $W$, as the constructions are intimately related to category $\OO$ of $\g$-modules. We define a module $\T_\ff$ of the Hecke algebra $\HH =\HH_W$ with a basis parametrized by $X_\ff$. This is carried out in Section~\ref{sec:module}. The module $\T_\ff$ can always be decomposed into a direct sum of permutation modules; this crucial property is not shared by the Hecke module used in  \cite{DJM98a, DJM98b, DS00}. The space $\T_\ff$ should be thought as a ``$q$-Grothendieck group" of a truncated version of the BGG category $\OO$, and the permutation module summands should be thought as ``blocks".

The $q$-Schur algebra $\Sc_q$  is by definition the algebra of $\HH$-linear endomorphisms of the $\HH$-module $\T_\ff$. In Section~\ref{sec:Schur}, we construct a (standard) $\A$-basis for the algebra $\Sc_q$. We establish the $q$-Schur duality that $\Sc_q$ and $\HH$ form double centralizers  in $\text{End}_\A(\T_\ff)$. The bar involutions on $\T_\ff$ and on $\Sc_q$ are defined, and the commuting actions of $\Sc_q$ and $\HH$ on $\T_\ff$ commute with the respective bar maps. By studying the property of the bar map acting on the standard basis, we construct a canonical basis for $\Sc_q$.

In \cite[Example~ 1.4]{Du94}, Du defined a particular $q$-Schur algebra of arbitrary finite type, and constructed its canonical basis. His example is a special case of our general algebraic constructions in this paper, and  its geometric and category $\mathcal O$ connections were not suspected or explained until now. (We thank Jie Du for bringing this  reference to our attention after we finished our work.)

According to Iwahori, the Hecke algebra $\HH$ can be realized via complete flags over finite fields. In Section~\ref{sec:geom}, we give a geometric construction of an $\A$-algebra $\Scg =\Scg(X_\ff)$ in terms of a collection of parabolic flag varieties, whose multiplicities of a given parabolic type are dictated by the set $X_\ff$. We construct an $(\Scg, \HH)$-bimodule $\T_\ff'$ geometrically via convolution products. We establish an algebra isomorphism $\Scg \cong \Sc_q$ and then a bimodule isomorphism $\T_\ff' \cong \T_\ff$. This provides a geometric realization of the $q$-Schur duality. Under the identification $\Scg \cong \Sc_q$, the canonical basis of $\Sc_q$ affords a geometric realization and hence admits some favorable positivity properties.

With suitable choices of the set $X_\ff$ in the classical type, our constructions reproduce the algebraic and geometric constructions in the papers mentioned earlier. Our approach is in turn a synthesis of the earlier DJ and BLM type approaches via purely Lie theoretic terms (such as weights, cosets, parabolic subgroups); it is not essential to use explicitly a tensor product module as in DJ's algebraic approach or ``$n$-step flags" as in BLM's geometric approach. The works of Du and Green \cite{Du92, Gr97} have been also very helpful in our understanding of the structures of $q$-Schur algebras.

The above $q$-constructions remain valid in the specialization at $q=1$, so we obtain a construction of Schur algebras and Schur duality of arbitrary type.

In Section~\ref{sec:G2}, we specialize to the type $G_2$ and study the $q$-Schur algebra of type $G_2$ in depth. Specifying the set $X_\ff =X_n$ (for integers $n\ge 2$), we denote the Schur algebra and its module by $\Sc_q(n)$ and $\T_n$, respectively.  We obtain a set of algebra generators (denoted by $e_a, f_a, t$ for some indices $a$) for the $\Q(q)$-Schur algebras. Let $\OO_n$ denote a version of the BGG category whose irreducible $G_2$-modules are parametrized by weights in $X_n$. Under a natural identification of the Grothendieck group $[\OO_n]$ with $\T_n|_{q=1}$, we establish an identification between the translation functors on $\OO_n$ with the generators $e_a, f_a, t$. We note that the idempotented (i.e., block) versions of $e_a, f_a, t$ are canonical bases of $\Sc_q(n)$.

Some detailed computations and formulas in type $G_2$ are collected in Appendix~\ref{sec:app}.

\subsection{}

We discuss below several interesting questions arising from this work.

For type BFG, the Hecke algebra $\HH$ admits a 2-parameter generalization. The 2-parameter $q$-Schur algebra of type B has been studied in \cite{Gr97}, and as explained in \cite{BWW18}, this is closely related to a unified approach toward the type B and type D Kazhdan-Ludztig theory \cite{BW18, Bao17}.  The $q$-Schur algebras and dualities in this paper can be formulated in the 2-parameter setting. It will be interesting to see if one can construct a canonical basis (when specifying the second parameter as an integrer power of $q$), though we will lose the positivity of the canonical basis in general and lose the connection to category $\OO$. It will be also interesting to study the cell and cellular algebra structures for the  (equal or unequal parameter) $q$-Schur algebras.

The precise connections with category $\OO$ of exceptional types (other than $G_2$) remain to be developed. Also it is interesting to find a presentation of the $q$-Schur algebras (over $\A$ or over $\Q(q)$).
Already in type $G_2$ it will be interesting to establish a generating set for the $q$-Schur algebra over $\A$ (not just over $\Q(q)$).  It remains to develop fully the connection between canonical bases of $q$-Schur algebras and translation functors; cf. Remark~\ref{rem:typeEF}.

In another direction, we hope our study of the type $G_2$ $q$-Schur algebras might  shed some light on a Kazhdan-Lusztig theory for the exceptional Lie superalgebra $G(3)$.


\vspace{2mm}
\noindent {\bf Acknowledgement.}
LL is partially supported by the Science and Technology Commission of Shanghai Municipality (grant No.
18dz2271000) and the NSF of China (grant No. 11671108, 11871214). WW is partially supported by the NSF grant DMS-1702254.
The starting point of this paper is the computations in type $G_2$, which were initiated during LL's visit to University of Virginia (UVA) in 2013-14. We thank both UVA and East China Normal University for hospitality and support.

\section{A Hecke module parametrized by weights}
  \label{sec:module}

In this section, we introduce a module of the Hecke algebra parametrized by a given $W$-invariant subset of integral weights.

\subsection{The preliminaries}

Let $\g$ be a complex simple Lie algebra of rank $d$ of arbitrary finite type, and let $\h\subset \g$ be a Cartan subalgebra.  Let $(\cdot,\cdot)$ denote the Killing form in $\h^*$, and $X$ be the weight lattice in $\h^*$. Fix a simple system $\Pi=\{\alpha_1,\ldots,\alpha_d\}$. The set of anti-dominant integral weights relative to $\Pi$ is denoted by $X^-$, i.e.,
\begin{equation}
 \label{eq:anti}
X^-=\{\mi\in X~|~(\mi,\alpha_i)\leq0, \forall i=1,2,\ldots,d\}.
\end{equation}

There is a partial order ``$\preceq$'' on $X$
defined by
\begin{equation}
\mi\preceq\mj \quad \Leftrightarrow \quad  \mj-\mi\in\sum_{i=1}^d\N\al_i.
\end{equation}


The Weyl group $W$ of $\g$ is generated by the simple reflections
$s_1,s_2,\ldots,s_d$ with identity $\id$. Let $m_{ij}$ denote the order of $s_is_j$ in $W$, for $i\neq j$.
The length of $w\in W$ is denoted by $\ell(w)$ and the Bruhat order on $W$ is denoted by ``$<$''.

Let $q$ be an indeterminate and let
\[
\A=\Z[q,q^{-1}].
\]
The Hecke algebra $\HH$ (associated to $W$) is an $\A$-algebra generated by $H_1,H_2,\ldots,H_d$ with relations
\begin{equation}
(H_i-q^{-1})(H_i+q)=0;\quad (H_iH_j)^{m_{ij}}=\id, \quad (1\leq i\neq j\leq d).
\end{equation}
For a reduced word $w=s_{i_1}s_{i_2}\cdots s_{i_l}\in W$, we set $H_w=H_{i_1}H_{i_2}\cdots H_{i_l}$.
Moreover, for any subset $Y\subseteq W$, we set
\begin{equation}\label{HY}
H_Y=\sum_{w\in Y}q^{-\ell(w)}H_w.
\end{equation}

\subsection{The Hecke modules $\T$ and $\T_{\ff}$}

There is a natural right action of $W$ on $X$ defined by sending $\mi \in X \mapsto \mi w=w^{-1}(\mi).$ Let us take any $W$-invariant finite subset
\begin{equation}\label{eq:Pf}
X_\ff \subset X.
\end{equation}
We shall discuss the choices of $X_\ff$ below in \S\ref{subsec:Pf}.


We introduce the following free $\A$-modules
\begin{equation}\label{eq:Tn}
\T =\bigoplus_{\mi\in X} \A v_\mi,
\qquad
 \T_{\ff} =\T_{X_\ff} = \bigoplus_{\mi\in X_{\ff}} \A v_\mi,
\end{equation}
with bases given by the symbols $v_\mi$, for $\mi\in X$ and $\mi\in X_{\ff}$, respectively.
We shall refer to $\{v_\mi\}$ as the {\em standard basis} for $\T$ or for $\T_{\ff}$.
We also define
\begin{equation}\label{ZformT}
\T^1=\bigoplus_{\mi\in X}\Z v_\mi, \quad \quad \T_{\ff}^1=\bigoplus_{\mi\in X_{\ff}}\Z v_\mi,
\end{equation}
which are the specializations at $q=1$ of $\T$ and $\T_{\ff}$, respectively.

The natural right action of $W$ on $X$ induces a right action of $W$ on $\T$ (and $\T_{\ff}$) by
\begin{equation}
v_{\mi}\cdot w=v_{\mi w}.
\end{equation}
Then we define a right action of the Hecke algebra $\HH$ on $\T$ (and $\T_{\ff}$) as follows:
\begin{align}\label{Haction}
v_{\mi}H_k&=
\left\{\begin{array}{ll}
q^{-1}v_{\mi}, & \mbox{if $\mi s_k=\mi$};\\
v_{\mi s_k}, & \mbox{if $\mi s_k\succ\mi$};\\
v_{\mi s_k}+(q^{-1}-q)v_{\mi}, & \mbox{if $\mi s_k\prec\mi$},
\end{array}\right. \quad\quad(1\leq k\leq d).
\end{align}

\begin{rem}
The spaces $\T^1$ (and respectively, $\T^1_{\ff}$) can be viewed as the Grothendieck groups of the category $\mathcal O$ (and respectively, a truncated version) of $\g$-modules of weights in $X$ (and respectively, $X_\ff$); see Section~\ref{sec:G2} for a precise formulation in type $G_2$.
\end{rem}

\subsection{Choices of $X_\ff$}
  \label{subsec:Pf}

There are natural choices for $X_\ff$; on the other hand, the choices are flexible and far from being unique.

In type $A_{d-1}$ (note we switch here from rank $d$ to $d-1$), we take $X =\sum_{i=1}^d \Z \vep_i$ to be the weight lattice for $GL(d)$, where $\{ \vep_i \}$ forms its standard basis. For any positive integer $n\ge d$, we can choose
\[
X_n = \big \{\sum_{i=1}^d a_i \vep_i~|~ a_i \in \Z, 1 \le a_i \le n, \forall i\big \}.
\]
We could have shifted the indices to consider ${}'X_n = \big \{\sum_{i=1}^d a_i \vep_i~|~ a_i \in \Z, 1- \lfloor {n\over 2} \rfloor \le a_i \le \lceil {n\over 2} \rceil, \forall i\big \}.$ The corresponding $\HH$-modules $\T_{X_n}$ and $\T_{{}'X_n}$ are isomorphic, and both can be naturally identified with $\mathbb V^{\otimes d}$ for a free $\A$-module $\mathbb V$ of rank $n$.
Note ${}'X_n \subset {}'X_{n+1}$ and  $X =\cup_n {}'X_n$.

In type $B_{d}$, $C_d$, or $D_d$, we let $X^0 =\sum_{i=1}^d \Z \vep_i$, and $X^{\hf} =\sum_{i=1}^d (\hf+\Z) \vep_i$, so that $X = X^{0} \oplus X^{\hf}$. For any positive integer $n$ with $2n\ge d$, we choose
\begin{align*}
X_n^{0} &= \big \{\sum_{i=1}^d a_i \vep_i~|~ a_i \in \Z, -n \le a_i \le n, \forall i\big \},
\\
X_n^{\hf} &= \big \{\sum_{i=1}^d a_i \vep_i~|~ a_i \in \hf+\Z, -n \le a_i \le n, \forall i\big \}.
\end{align*}
The $\HH$-modules $\T_{X_n^0}$ and $\T_{X_n^{\hf}}$ can be identified with $\mathbb V^{\otimes d}$ for a free $\A$-module $\mathbb V$ of rank $2n+1$ and $2n$, respectively.
Note that $X_n^{0} \subset X_{n+1}^{0}$ and  $X^{0} =\cup_n X_n^{0}$;  similarly, we have $X_n^{\hf} \subset X_{n+1}^{\hf}$ and  $X^{\hf} =\cup_n X_n^{\hf}$.

For type $G_2$, we have some natural choices for $X_\ff$; see Section~\ref{sec:G2}.

For types $F_4$, or $E_{6,7,8}$, again we can choose suitable finite subsets $X_n$ (for varied $n$) with $X_n \subset X_{n+1}$ and $X=\cup_{n} X_n$. As we do not study their connections to category $\OO$ later on,  there is no need to get into the detailed choices here.

The choices in classical types above and in $G_2$ are motivated by considerations in the BGG category $\OO$ and related geometric setting via flag varieties.  As we shall see, the above choices in classical types are compatible with various earlier algebraic, geometric and categorical constructions in the literature; see \cite{DJ89, BLM90} for type A,  \cite{Gr97, BW18, BKLW18} for type B/C, and \cite{FL15, Bao17} for type D.

\section{The $q$-Schur algebras and dualities}
 \label{sec:Schur}

In this section, we introduce the $q$-Schur algebra $\Sc_q$ associated to a given $W$-invariant finite set of integral weights, and establish the $q$-Schur duality between the $q$-Schur algebra $\Sc_q$ and the Hecke algebra. We then construct several variations of the standard basis for $\Sc_q$, and then we prove the existence of the canonical basis for $\Sc_q$.

\subsection{The definition and the first properties}

Fix $X_{\ff}$ in \eqref{eq:Pf} and recall $\T_{\ff}$ from \eqref{eq:Tn}.
Introduce the following algebras (over $\A$ and $\Q(q)$, respectively):
\[
\Sc_q = \Sc_q(X_\ff) :=\End_{\HH}(\T_{\ff}), \qquad \Sc_{q,\Q} :=\Q(q)\otimes_\A \Sc_q.
\]
The algebras $\Sc_{q}$ and $\Sc_{q,\Q}$ will be called $q$-Schur algebras; they depend on the choices of $W$ and $X_\ff$.

For $\mi,\mj\in X_{\ff}$, we define the coordinate functions
\[
c_{\mi,\mj}: \End_\A(\T_{\ff}) \rightarrow \A, \qquad \eta \mapsto c_{\mi,\mj}(\eta),
\]
via
\begin{equation}
\eta(v_{\mj})=\sum_{\mi\in X_{\ff}}c_{\mi,\mj}(\eta)v_{\mi}.
\end{equation}

Let us describe the elements in $\Sc_{q}$ in terms of the coordinate functions.

\begin{lem}
 \label{lem:coord}
Let $\eta\in \End_\A(\T_{\ff})$. Then we have $\eta\in \Sc_{q}$ if and only if
\begin{equation}\label{eq:cij}
c_{\mi,\mj}(\eta)=\left\{\begin{array}{ll}q^{-1} c_{\mi s_k,\mj}(\eta) & \mbox{ if } \mi s_k \prec \mi, \mj s_k=\mj,\\
q^{-1} c_{\mi,\mj s_k}(\eta) & \mbox{ if } \mi s_k = \mi, \mj s_k\prec\mj,\\
c_{\mi s_k,\mj s_k}(\eta) & \mbox{ if } \mi s_k\succ \mi, \mj s_k\prec\mj,\\
c_{\mi s_k,\mj s_k}(\eta)+(q^{-1}-q)c_{\mi,\mj s_k}(\eta) & \mbox{ if } \mi s_k \prec \mi, \mj s_k\prec\mj\\
\end{array}\right.
\end{equation}
for any $\mi,\mj\in X_{\ff}$ and $1\leq k\leq d$.
\end{lem}
\begin{proof}
By definition, we have $\eta\in \Sc_{q}$ if and only if the action of $\eta$ on $\T_{\ff}$ commutes with the action of the generators $H_k (1\leq k\leq d)$ of $\HH$.
Write $c_{\mi,\mj}=c_{\mi,\mj}(\eta)$ for short. We calculate that
\begin{align*}
(\eta & (v_\mj))H_k =(\sum_{\mi\in X_{\ff}}c_{\mi,\mj}v_{\mi})H_k\\
&=
\sum_{\mi\in X_{\ff}, \mi s_k = \mi}c_{\mi,\mj}q^{-1}v_{\mi}+\sum_{\mi\in X_{\ff}, \mi s_k \succ \mi}c_{\mi,\mj}v_{\mi  s_k}
+\sum_{\mi\in X_{\ff}, \mi s_k \prec \mi}c_{\mi,\mj}(v_{\mi  s_k}+(q^{-1}-q)v_{\mi})\\
&=
\sum_{\mi\in X_{\ff}, \mi s_k = \mi}c_{\mi,\mj}q^{-1}v_{\mi}+\sum_{\mi\in X_{\ff}, \mi s_k \succ \mi}c_{\mi  s_k,\mj}v_{\mi}
+\sum_{\mi\in X_{\ff}, \mi s_k \prec \mi}(c_{\mi s_k,\mj}+c_{\mi,\mj}(q^{-1}-q))v_{\mi},
\end{align*}
and
\begin{eqnarray*}
\eta((v_\mj)H_k)=\left\{\begin{array}{ll}\eta(q^{-1}v_{\mj})=
\sum_{\mi\in X_{\ff}}c_{\mi,\mj}q^{-1}v_{\mi}, & \mbox{if $\mj s_k=\mj$,}\\
\eta(v_{\mj s_k})=
\sum_{\mi\in X_{\ff}}c_{\mi,\mj s_k}v_{\mi}, & \mbox{if $\mj s_k\succ\mj$,}\\
\eta(v_{\mj s_k}+(q^{-1}-q)v_{\mj})=\sum_{\mi\in X_{\ff}}(c_{\mi,\mj s_k}+(q^{-1}-q)c_{\mi,\mj})v_{\mi},&\mbox{if $\mj s_k\prec\mj$}.
\end{array}\right.
\end{eqnarray*}
Comparing the coefficients for $(\eta(v_\mj))H_k$ and $\eta((v_\mj)H_k)$ proves the lemma.
\end{proof}

Denote by $X_{\ff}^{(2)}$ the subset of $X_{\ff}\times X_{\ff}$ consisting of elements $(\mi,\mj)$
satisfying the following conditions:
\begin{itemize}
\item[(i)] $\mj\in X_{\ff}^-$;

\item[(ii)] For any $1\leq k\leq d$, if $\mj s_k=\mj$ then $\mi s_k\succ\mi$.
\end{itemize}
Note that $X_{\ff}^{(2)}$ is a transversal for the $W$-orbits on $X_{\ff}\times X_{\ff}$.

For each $(\mi,\mj)\in X_{\ff}^{(2)}$, we define the element $\eta_{\mi,\mj}\in \End_\A(\T_{\ff})$ by
\begin{equation}\label{def:eta}
c_{\mk,\ml}(\eta_{\mi,\mj})=
\left\{\begin{array}{ll} 1& \mbox{if $(\mi,\mj)= (\mk,\ml)$,}\\
0 & \mbox{otherwise,}
\end{array}\right. \quad \forall(\mk,\ml)\in X_{\ff}^{(2)}.
\end{equation}

Define $\Sc_{q}^*:=\Hom_\A(\Sc_{q},\A)$, which is naturally an $\A$-module.
\begin{prop}
  \label{schurbasis}
The $\A$-module $\Sc_{q}^*$ is free with $\{c_{\mi,\mj}\mid (\mi,\mj)\in X_{\ff}^{(2)}\}$ as a basis.
The $\A$-module $\Sc_{q}$ is free with a basis $\{\eta_{\mi,\mj}\mid (\mi,\mj)\in X_{\ff}^{(2)}\}$.
\end{prop}

\begin{proof}
Since $\{\eta_{\mi,\mj}\mid (\mi,\mj)\in X_{\ff}^{(2)}\}$ is dual to $\{c_{\mi,\mj}\mid (\mi,\mj)\in X_{\ff}^{(2)}\}$ by definition, the second statement follows from the first one.

Let us prove the first statement.
An arbitrary element $c_{\mk,\ml}\in \Sc_{q}^*$ can be expressed as a sum of elements $c_{\mi,\mj}$ such that $\mj\in X_{\ff}^-$, thanks to \eqref{eq:cij}.
We then apply the first equation in \eqref{eq:cij} to these $c_{\mi,\mj}$ until condition (ii) in the definition of $X_{\ff}^{(2)}$ is satisfied as well. Therefore, the set $\{c_{\mi,\mj}\mid (\mi,\mj)\in X_{\ff}^{(2)}\}$ spans $\Sc_{q}^*$.

It remains to show that $\{c_{\mi,\mj}~|~ (\mi,\mj)\in X_{\ff}^{(2)}\}$ are linearly independent over $\A$. Suppose
\begin{equation}\label{indep}
\sum_{(\mi,\mj)\in X_{\ff}^{(2)}}\alpha_{\mi,\mj}c_{\mi,\mj}=0\quad\mbox{for some $\alpha_{\mi,\mj}\in\A$ which are not all zero}.
\end{equation}
Multiplying $q$-powers if necessary, we may assume that all $\alpha_{\mi,\mj}$ lie in $\mathbb{Z}[q]$ and have no common divisors. Hence $\alpha_{\mi,\mj}|_{q=1}$ are not all zero. 

Consider the specialization at $q=1$. Equation \eqref{indep} becomes
$$\sum_{(\mi,\mj)\in X_{\ff}^{(2)}}\alpha_{\mi,\mj}c_{\mi,\mj}=0\quad\mbox{for some $\alpha_{\mi,\mj}\in \mathbb{Z}$ which are not all zero.}$$
Now \eqref{eq:cij} at $q=1$ says that $c_{\mi,\mj}(\eta)=c_{\mi w,\mj w}(\eta)$ for $w\in W$
and $\eta\in\End_{W}(\T_{\ff}^1)$ (recall $\T_{\ff}^1$ in \eqref{ZformT}). Hence the endomorphism $\eta_{\mk,\ml}^1$ of $\T_{\ff}^1$, for $(\mk,\ml)\in X_{\ff}^{(2)}$, defined by
$$c_{\mi,\mj}(\eta_{\mk,\ml}^1)=\left\{
\begin{array}{ll}
1,& (\mi,\mj)\sim (\mk,\ml);\\
0,& \mbox{otherwise},
\end{array}\right.$$
does lie in $\End_{W}(\T_{\ff}^1)$.
Therefore we obtain
$$0=\sum_{(\mi,\mj)\in X_{\ff}^{(2)}}\alpha_{\mi,\mj}c_{\mi,\mj}(\eta_{\mk,\ml}^1)=\alpha_{\mk,\ml}, \quad \forall(\mk,\ml)\in X_{\ff}^{(2)},$$
which is a contradiction to \eqref{indep}.
Thus  $\{c_{\mi,\mj}~|~ (\mi,\mj)\in X_{\ff}^{(2)}\}$ are independent.

The proposition is proved.
\end{proof}

\subsection{Another view of the $q$-Schur algebra}

In this subsection, we shall identify $\T_{\ff}$ as a sum of permutation modules of $\HH$, and then study the $q$-Schur algebra $\Sc_{q}$ accordingly. This will be helpful for the double centralizer property and the construction of canonical basis on $\Sc_{q}$ in the next subsections.

Note in each $W$-orbit (called a {\em linkage class}) in $X$ there exists a unique anti-dominant element in $X^-$; cf. \eqref{eq:anti}.
Denote
\begin{align}
  \label{eq:Lambda}
\Lambda &=\{\mbox{linkage classes in $X$}\},
\\
\Lambda_{\ff} &=\{\mbox{linkage classes in $X_{\ff}$}\},
\label{Ln}
\\
\bu_\gamma & =\mbox{ the unique anti-dominant element in a linkage class } \gamma \in \Lambda.
\label{eq:ibu}
\end{align}
In other words, there is a bijection
\begin{equation}
  \label{eq:XnLn}
\Lambda_{\ff} \leftrightarrow X_{\ff}^-,
\qquad \gamma \mapsto \bu_\gamma.
\end{equation}

For any subset $J \subseteq \{1,2,\ldots,d\}$, let $W_J$ be the parabolic subgroup of $W$ generated by $\{ s_j \mid j \in J \}$.
Let $\HH_J$ be the subalgebra of $\HH$ generated by $\{ H_j \mid j \in J \}$,
and $\D_J$ be the set of minimal length right coset representatives for $W_J \setminus W$.

For any $\gamma\in\Lambda_{\ff}$, we define the subset
\begin{equation}\label{eq:PP}
J_\gamma=\{k~|~ 1\leq k\leq d, \bu_\gamma s_k=\bu_\gamma\}.
\end{equation}
We shall write $W_\gamma =W_{J_\gamma}$ and $\D_\gamma=\D_{J_\gamma}$.
We denote by $w_\circ^J$ and $w_\circ^\gamma$ the unique longest element in $W_J$ and $W_\gamma$, respectively.
The subspace of $\T_{\ff}$
\[
\T_\gamma :=\bigoplus_{\mi\in\gamma}\A v_{\mi}
\]
is clearly a right $\HH$-module, and thus we have the following $\HH$-module decomposition
\[
\T_{\ff}=\bigoplus_{\gamma\in\Lambda_{\ff}}\T_\gamma.
\]

For $J\subseteq \{1,2,\ldots,d\}$ and $\gamma \in \Lambda_{\ff}$, we define the $q$-symmetrizers $x_J, x_\gamma \in\HH$ by
\begin{align*}
x_J & =\sum_{w\in W_J} q^{\ell(w_\circ^J)-\ell(w)}H_w,
\\
x_\gamma &=x_{J_\gamma}.
\end{align*}
(Our convention for $x_J$ here differs from some literature by a factor $q^{\ell(w_\circ^J)}$.)

\begin{lem}
  \label{lem:iden}
For each $\gamma \in \Lambda_{\ff}$, there exists a right $\HH$-module isomorphism
\[
\Omega_\gamma: \T_\gamma \stackrel{\cong}{\longrightarrow} x_\gamma\HH,
\qquad v_{\bu_\gamma} \mapsto x_\gamma.
\]
Moreover, $\Omega_\gamma( v_{\bu_\gamma w} ) =x_\gamma H_w$, for $w \in \D_\gamma$.
This induces an identification
\[
\Omega: \T_{\ff} \cong \bigoplus_{\gamma\in\Lambda_{\ff}}x_\gamma\HH.
\]
\end{lem}

\begin{proof}
Letting $\Omega_\gamma( v_{\bu_\gamma  w} ) =x_\gamma H_w$, for $w \in \D_\gamma$, clearly defines an $\A$-module isomorphism $\Omega_\gamma: \T_\gamma \stackrel{\cong}{\longrightarrow} x_\gamma\HH$. One then checks that this map commutes with the actions of $H_k$ ($1\leq k\leq d$) given by \eqref{Haction}, using the well known formulas for $H_w H_i$. Therefore, $\Omega_\gamma$ is a right $\HH$-module isomorphism. Clearly the $\HH$-module homomorphism $\Omega_\gamma$ is determined by the formula $\Omega_\gamma( v_{\bu_\gamma}) =x_\gamma$.
\end{proof}

Thanks to Lemma~\ref{lem:iden} above, we can identify
\begin{equation*}
\Sc_{q}=\End_{\HH}(\oplus_{\gamma\in\Lambda_{\ff}}x_\gamma\HH)=\bigoplus_{\gamma,\nu\in\Lambda_{\ff}}\Hom_{\HH}(x_\nu\HH,x_\gamma\HH).
\end{equation*}

Explicitly, the set $\D_\gamma$ of distinguished minimal length right coset representatives of $W_\gamma$ in $W$ is
\[
\D_\gamma=\{g\in W~|~ \ell(wg)=\ell(w)+\ell(g), \forall w\in W_\gamma\}.
\]
It is known that $\D_{\gamma}^{-1}$ is the set of distinguished minimal length left coset representatives of $W_\gamma$ in $W$.
Denote by
\[
\D_{\gamma\nu}=\D_{\gamma}\cap\D_\nu^{-1}
\]
the set of minimal length double coset representatives.
For $\gamma,\nu\in\Lambda_{\ff}$ and $g\in\D_{\gamma\nu}$, recalling the notation \eqref{HY}, we denote by
\begin{align}  \label{eq:phig}
\begin{split}
\phi_{\gamma\nu}^g\in \Sc_{q} &=\End_{\HH}(\oplus_{\gamma\in\Lambda_{\ff}}x_\gamma\HH),
\\
x_{\nu'}  & \mapsto  \delta_{\nu,\nu'} q^{\ell(w_\circ^\nu)}H_{W_\gamma gW_\nu}, \quad \forall \nu' \in \Lambda_{\ff}.
\end{split}
\end{align}
Let
\begin{equation}   \label{eq:Xi}
\Xi=\{(\gamma,g,\nu)~|~\gamma,\nu\in\Lambda_{\ff}, g\in\D_{\gamma\nu}\}.
\end{equation}

\begin{lem}\label{bijection}
There is a bijective map $\Xi \longrightarrow X_{\ff}^{(2)}, \; (\gamma,g,\nu)\mapsto (\bu_\gamma g,\bu_\nu).$
\end{lem}

\begin{proof}
The anti-dominant weights $\bu_\gamma$ and $\bu_\nu$ are determined by $\gamma$ and $\nu$, respectively.
If $\bu_\gamma g=\bu_\gamma h$ $(g,h\in\D_{\gamma\nu})$, then $g\in W_\gamma h$ and hence $g=h$.
So it is injective.

For any $(\mi,\mj)\in X_{\ff}^{(2)}$, we can read off the linkage class $\gamma$ (resp. $\nu$) of $\mi$ (resp. $\mj$) immediately. Let $g\in\D_{\gamma}$ be of minimal length such that $\mi=\bu_\gamma g$.
Condition (ii) in the definition of $X_{\ff}^{(2)}$ forces that $g\in\D_\nu^{-1}$. Thus it is surjective.
\end{proof}

\begin{prop}  \label{prop:basis}
We have the following identification of endomorphisms on $\T_{\ff}=\oplus_{\gamma\in\Lambda_{\ff}}x_\gamma\HH$:
\begin{equation}\label{eq:phieta}
\phi_{\gamma\nu}^g= q^{\ell(w_\circ^\nu)-\ell(w_\circ^\gamma)-\ell(g)} \eta_{\bu_\gamma g,\bu_\nu}.
\end{equation}
Hence $\{\phi_{\gamma\nu}^g~|~\gamma,\nu\in\Lambda_{\ff}, g\in\D_{\gamma\nu}\}$ is an $\A$-basis of $\Sc_{q}$.
\end{prop}

\begin{proof}
Note the identification $x_\gamma=v_{\bu_\gamma}$ given in Lemma~\ref{lem:iden}. For any $(\gamma,g,\nu)\in\Xi$, we define the subset $\mathcal{C}(\gamma,g,\nu)$ of $W$ to be the set of minimal length right coset representatives for $W_\gamma$ in $W$ which lie in the double coset $W_\gamma gW_\nu$. It follows from the definition that $\mathcal{C}(\gamma,g,\nu)\subset g W_\nu.$ Thus we can write an element $w\in\mathcal{C}(\gamma,g,\nu)$ in the form $w=gy$ with $y\in W_\nu$.
We compute that
\begin{align}\label{actionphi}
\phi_{\gamma\nu}^g(v_{\bu_\nu})&=q^{\ell(w_\circ^\nu)}H_{ W_\gamma g W_\nu}=q^{\ell(w_\circ^\nu)-\ell(w_\circ^\gamma)}\sum_{w\in \mathcal{C}(\gamma,g,\nu)}q^{-\ell(w)}v_{\bu_{\gamma}}H_w
\\&=q^{\ell(w_\circ^\nu)-\ell(w_\circ^\gamma)}\sum_{w\in \mathcal{C}(\gamma,g,\nu)}q^{-\ell(w)}v_{\bu_{\gamma} w}. \nonumber
\end{align}
On the other hand, we have
\begin{equation}\label{actioneta}
\eta_{\bu_\gamma g,\bu_\nu}(v_{\bu_\nu})=\sum_{gy\in\mathcal{C}(\gamma,g,\nu)}c_{\bu_\gamma gy,\bu_\nu}(\eta_{\bu_\gamma g,\bu_\nu})v_{\bu_\gamma gy}
=\sum_{gy\in\mathcal{C}(\gamma,g,\nu)}q^{-\ell(y)}v_{\bu_\gamma gy},
\end{equation}
where the first equality uses  $\bu_\nu y=\bu_{\nu}$ and the second equality uses \eqref{eq:cij}.

Comparing \eqref{actionphi} and \eqref{actioneta} gives us
\[
\phi_{\gamma\nu}^g(v_\mi)=q^{\ell(w_\circ^\nu)-\ell(w_\circ^\gamma)-\ell(g)}\eta_{\bu_\gamma g,\bu_\nu}(v_{\mi}),\quad \forall i\in X_{\ff}^-.\]
By the commuting action of $\HH$, this implies that
\begin{align}\label{eq:phieta1}
\phi_{\gamma\nu}^g(v_\mi H)
&=q^{\ell(w_\circ^\nu)-\ell(w_\circ^\gamma)-\ell(g)}\eta_{\bu_\gamma g,\bu_\nu}(v_{\mi}H),\quad \forall \mi\in X_{\ff}^-, H\in\HH.
\end{align}
Since $\T_{\ff}$ is generated as an $\HH$-module by $\{v_{\mi}~|~\mi\in X_{\ff}^-\}$, the equality \eqref{eq:phieta} follows from \eqref{eq:phieta1}.

Finally, the second statement follows by \eqref{eq:phieta}, Lemma \ref{bijection} and Proposition \ref{schurbasis}.
\end{proof}

\subsection{A $q$-Schur duality}

We first prepare some lemmas. Write $\mi\sim\mj$ if $\mi$ and $\mj$ are in the same linkage class.

\begin{lem}
If $\mi\in X_{\ff}^-$, then $\eta_{\mi,\mi}\in \Sc_{q}$ is an idempotent such that
\begin{equation*}
\eta_{\mi,\mi} v_{\mj}=\left\{\begin{array}{ll}
v_{\mj} & \mbox{if $\mj\sim\mi$},\\
0 & \mbox{otherwise}.
\end{array}\right.
\end{equation*}
\end{lem}

\begin{proof}
Consider the element $\eta\in\End_\A(\T_{\ff})$ whose coordinate functions are given by
\begin{equation*}
c_{\mk,\ml}(\eta)=\left\{\begin{array}{ll}
1 & \mbox{if $\mk=\ml\sim \mi$},\\
0 & \mbox{otherwise},
\end{array}\right.
\end{equation*}
for all $\mk, \ml \in X_{\ff}$. Then $\eta$ satisfies all the relations in \eqref{eq:cij}, and so we have $\eta \in \Sc_{q}$. The only element $c_{\mk,\ml}$ with $(\mk,\ml)\in X_{\ff}^{(2)}$ which does not annihilate $\eta$ is $c_{\mi,\mi}$. Thus $\eta=\eta_{\mi,\mi}$ since $c_{\mi,\mi}(\eta)=1=c_{\mi,\mi}(\eta_{\mi,\mi})$. The lemma follows.
\end{proof}

\begin{lem}  \label{lem:map}
Let $\psi\in\End_{\Sc_{q}}(\T_{\ff})$, $\mi \in X_{\ff}$, and write
$\psi(v_{\mi})=\sum_{\mj}\alpha_{\mj}v_{\mj}$.
Then $\alpha_{\mj}\neq 0$ implies $\mj\sim\mi$.
\end{lem}

\begin{proof}
Take any $\mk\not\sim\mi$. We have
\begin{equation}
0=\psi(\eta_{\mk,\mk}v_{\mi})=\eta_{\mk,\mk}(\psi v_{\mi})
=\eta_{\mk,\mk}(\sum_{\mj}\alpha_{\mj}v_{\mj})=\sum_{\mj\sim\mk}\alpha_{\mj}v_{\mj}.
\end{equation}
Therefore, we have $\alpha_{\mj}=0$ if $\mj\sim\mk\not\sim\mi$.
\end{proof}

%

Let us denote the right action of $\HH$ on $\T_{\ff}$ by $\Psi$. This action $\Psi$ is faithful if $X_\ff$ contains at least one regular $W$-orbit, where ``regular'' means that the orbit has cardinality $|W|$.  

\begin{thm} [$q$-Schur duality]
 \label{thm:double}
Suppose $X_\ff$ contains at least one regular $W$-orbit. The algebras $\Sc_{q}$ and $\HH$ satisfy the double centralizer property, i.e.,
\begin{align}
  \label{eq:duality}
\begin{split}
\Sc_{q} =\;& \End_{\HH}(\T_{\ff}),
\\
& \End_{\Sc_{q}}(\T_{\ff}) =\Psi(\HH).
\end{split}
\end{align}
Moreover, we have $\Psi(\HH) \cong \HH$.
\end{thm}

\begin{proof}
The first statement and the inclusion $\End_{\Sc_{q}}(\T_{\ff}) \supseteq \Psi(\HH)$ are clear from definitions.
Let $\psi \in \End_{\Sc_{q}}(\T_{\ff})$. Since $X_{\ff}$ contains at least one regular $W$-orbit, we can fix a regular class $\omega\in\Lambda_{\ff}$ (hence $J_\omega=\emptyset$ and $W_{\omega}=\{\id\}$). Then $\phi^g_{\gamma \omega} \cdot x_\omega = q^z x_\gamma H_g$ (for some $z\in \Z$), for any $\gamma \in \Lambda_{\ff}$. It follows by Lemma~ \ref{lem:map} that each endomorphism $\psi\in\End_{\Sc_{q}}(\T_{\ff})$ maps each subspace $\T_{\gamma}$ to itself. Note that $\psi$ is completely determined by its value on $\psi(x_\omega)$, thanks to $\phi^g_{\gamma \omega} \in \Sc_{q}$ by Proposition~\ref{prop:basis} and hence $\psi(x_\gamma H_g) =q^{-z} \psi(\phi^g_{\gamma \omega} \cdot x_\omega) =q^{-z} \phi^g_{\gamma \omega} \cdot \psi ( x_\omega)$. Therefore we have a natural identification
$\End_{\Sc_{q}}(\T_{\ff}) \cong \T_\omega$, $\psi \mapsto \psi(x_\omega)$. Combining this with another natural identification $\Psi(\HH) \cong \T_\omega$, $\Psi(h)\mapsto x_\omega h =\Psi(h) (x_\omega)$, we have proved that $\End_{\Sc_{q}}(\T_{\ff}) =\Psi(\HH)$.

Since $\T_{\ff}$  contains a regular representation of $\HH$, the right action of $\HH$ on $\T_{\ff}$ is faithful, i.e., $\Psi(\HH) \cong \HH$.
\end{proof}

\begin{example}
 \label{ex:S=H}
If we take $X$ to be a single regular orbit of $W$, the associated $q$-Schur algebra is isomorphic to the Hecke algebra $\HH$, and $\T_{\ff}$ is the regular $\HH$-bimodule.
\end{example}

\begin{rem}
We suspect Theorem~\ref{thm:double} remains valid after the removal of the condition ``$X_\ff$ contains at least one regular $W$-orbit", and this calls for some new argument.
\end{rem}

\begin{rem}
The notion of $q$-Schur algebra and the $q$-Schur duality admit classical counterparts, which can be obtained by taking the $q\mapsto 1$ limit.
\end{rem}

\subsection{The canonical basis for $\T_{\ff}$}

The bar involution on $\HH$ is the $\Z$-algebra automorphism defined by
$\overline{H_k}=H_k^{-1}\ (1\leq k\leq d)$ and $\bar{q}=q^{-1}$. Recall the Bruhat order ``$<$'' on $W$.
Let  $J \subseteq \{ 1,2,\ldots,d \}$. Then, for each $w \in\D_J$, there exists (cf. \cite{KL79, Deo87}) a unique element $C^J_w \in x_J \HH$ such that
\begin{enumerate}
\item $\overline{C^J_w} = C^J_w$,
\item
$
C^J_w \in  x_J \Big(H_w + \sum_{{y \in\D_J, \, y < w}} q \Z[q] H_y \Big).$
\end{enumerate}
Moreover, the elements $\{C^J_w~ \vert~ w \in\D_J\}$ forms a $\Q(v)$-basis of $x_J \HH$ (called the canonical basis or parabolic KL basis).
In case $J=\emptyset$ and hence $W_\emptyset=\{\id \}$, we are back to the original setting of Kazhdan-Lusztig and  shall write $C_w=C_w^\emptyset$.

Since the $q$-symmetrizer $x_\gamma$ is bar invariant (and it is equal to $C_{w_o^\gamma}$ in $\HH$),  $x_\gamma\HH$ is a bar invariant right $\HH$-submodule of $\HH$. The identifications
$\Omega_\gamma: \T_\gamma \cong x_\gamma\HH$ $(v_{\bu_\gamma} \mapsto x_\gamma)$ and $\Omega: \T_{\ff}\cong \oplus_{\gamma\in\Lambda_{\ff}}x_\gamma\HH$ in Lemma~\ref{lem:iden} induces a bar involution on $\T_\gamma$, for all $\gamma \in \Lambda_{\ff}$, and on $\T_{\ff}$.
More explicitly, the bar involution on the $\A$-module $\T$ (or $\T_{\ff}$) can be characterized by
\begin{eqnarray}
\label{baranti}
\begin{split}
\overline{v_{\mi}} &=v_{\mi},\quad\quad \mbox{for $\mi\in X^-$,}
\\
\overline{v_{\mi}h} &=\overline{v_{\mi}}\overline{h},
\quad\mbox{for  $\mi\in X$ (or $X_{\ff}$), and $h\in\HH$}.
\end{split}
\end{eqnarray}
Note by \eqref{baranti} that $v_{\bu_\gamma}$  are bar invariants.

By the identification $\Omega: \T_{\ff}\cong \oplus_{\gamma\in\Lambda_{\ff}}x_\gamma\HH$, the $\A$-module $\T_{\ff}$ admits a canonical basis
\[
\B (\T_{\ff}) =\Big\{C_{\bu_\gamma  w} := \Omega^{-1}(C_w^{J_\gamma}) ~\big |~\gamma \in \Lambda_{\ff}, ~ w \in\D_\gamma \Big\}.
\]
The canonical basis $\B(\T_{\ff})$ on $\T_{\ff}$ can be characterized by the following two properties:
\begin{enumerate}
\item $\overline{C_{\mi}} = C_{\mi}$, for $\mi \in X_{\ff}$;
\item
$C_{\mi} \in  v_{\mi} + \sum_{{y \in\D_\gamma, \, y < w}} q \Z[q] v_{\bu_\gamma  y}$, for $\mi =\bu_\gamma  w$ with  $\gamma \in \Lambda_{\ff},w \in \D_\gamma$.
\end{enumerate}

\subsection{The canonical basis for $\Sc_{q}$}
  \label{secbar}

For $(\gamma,g,\nu)\in\Xi$, set $g_{\gamma\nu}^+$ to be the longest element in $ W_\gamma g W_\nu$. In particular, $\id_{\nu\nu}^+=w_\circ^\nu$ is the longest element in $ W_\nu$.

\begin{lem}\label{eq:xandC}
Let $(\gamma,g,\nu)\in\Xi$. Then we have
\begin{itemize}
\item[(1)] $ W_\gamma g  W_\nu = \{w \in W ~|~ g \leq w \leq g^+_{\gamma\nu}\}$;

\item[(2)] $H_{ W_\gamma g  W_\nu}
= q^{-\ell(g^+_{\gamma\nu})} C_{g^+_{\gamma\nu}} + \sum_{\substack{y\in \D_{\gamma\nu}\\
 y < g }} c^{(\gamma,\nu)}_{y,g} C_{y^+_{\gamma\nu}}$,
for $c^{(\gamma,\nu)}_{y,g}\in\A.$
\end{itemize}
\end{lem}
\begin{proof}
See \cite{Cur85}.
\end{proof}

We define a bar involution $\bar{\phantom{x}}$ on $\Sc_{q}$ as follows: for each $f \in \Hom_{\HH}(x_\nu \HH, x_\gamma \HH) \subset \Sc_{q}$, let $\overline{f}\in\Hom_{\HH}(x_\nu \HH, x_\gamma \HH) \subset \Sc_{q}$ be the
$\HH$-linear map which sends $x_\nu=C_{w_\circ^\nu}$ to $\overline{f(C_{w_\circ^\nu})}$. That is, we have
\begin{equation}
    \label{eq:barS}
\overline{f}(x_{\nu'} h) = \delta_{\nu',\nu} \overline{f(x_\nu)} h,
\quad
\mbox{for $h \in \HH$}.
\end{equation}
Hence it follows from Lemma \ref{eq:xandC} that
\begin{eqnarray}
\phi_{\gamma\nu}^g(C_{w_\circ^\nu})
&=& q^{\ell(w_\circ^\nu)-\ell(g^+_{\gamma\nu})} C_{g^+_{\gamma\nu}}
 + \sum_{\substack{y\in\D_{\gamma\nu}\\y < g}}
q^{\ell(w_\circ^\nu)} c_{y,g}^{(\gamma,\nu)} C_{y^+_{\gamma\nu}},
    \label{eq:eA}
\\
\overline{\phi_{\gamma\nu}^g}(C_{w_\circ^\nu})
&=& q^{\ell(g^+_{\gamma\nu})-\ell(w_\circ^\nu)} C_{g^+_{\gamma\nu}}
+ \sum_{\substack{y\in\D_{\gamma\nu}\\y < g}}
q^{-\ell(w_\circ^\nu)}\overline{c_{y,g}^{(\gamma,\nu)}} C_{y^+_{\gamma\nu}}.
    \label{eq:eAbar}
\end{eqnarray}

For any $(\gamma,g,\nu)\in\Xi$, we set
\begin{equation}\label{def:dA}
[\phi_{\gamma\nu}^g]=q^{\ell(g_{\gamma\nu}^+)-\ell(w_\circ^\nu)}\phi_{\gamma\nu}^g.
\end{equation}
Then $\{[\phi_{\gamma\nu}^g] \mid (\gamma,g,\nu)\in\Xi \}$ forms an $\A$-basis for $\Sc_q$, which is called a {\em standard basis}. Thanks to \eqref{eq:eA} and \eqref{eq:eAbar}, we have
\begin{equation}\label{barbraket}
\overline{[\phi_{\gamma\nu}^g]}\in [\phi_{\gamma\nu}^g]+\sum_{g>y\in\D_{\gamma\nu}}\A[\phi_{\gamma\nu}^y].
\end{equation}

Similar to \cite{Du92}, we define
\begin{equation}\label{canobasis}
   \{\phi_{\gamma\nu}^g\} \in  \Hom_\HH(x_\nu\HH, x_\gamma\HH), \text{  and hence } \{\phi_{\gamma\nu}^g\} \in \Sc_{q},
\end{equation}
   by requiring
\begin{equation}
  \label{canonicalbasis}
  \{\phi_{\gamma\nu}^g\} (C_{w^{\nu}_\circ}) = C_{g^+_{\gamma\nu}}.
\end{equation}
It follows by \eqref{eq:barS} that $\{\phi_{\gamma\nu}^g\}$ is bar invariant, i.e.,
\begin{equation}\label{barA}
     \overline{\{\phi_{\gamma\nu}^g\}} = \{\phi_{\gamma\nu}^g\}.
\end{equation}

Following \cite[(2.c), Lemma~3.8]{Du92}, we have
\begin{equation}   \label{eq:canonical}
\{\phi_{\gamma\nu}^g\} \in [\phi_{\gamma\nu}^g] + \sum_{y< g}q\Z[q]   \, [\phi_{\gamma\nu}^y].
\end{equation}
More precisely, we have
\begin{equation}   \label{eq:CB1}
\{\phi_{\gamma\nu}^g\} =  [\phi_{\gamma\nu}^g] + \sum_{y< g}  q^{\ell(g^+_{\gamma\nu})-\ell(y^+_{\gamma\nu})} P_{y^+_{\gamma\nu},g^+_{\gamma\nu}} [\phi_{\gamma\nu}^y].
\end{equation}
where $P_{y^+_{\gamma\nu},g^+_{\gamma\nu}}$ are Kazhdan-Lusztig polynomials.

By Proposition~ \ref{prop:basis}, \eqref{def:dA} and \eqref{eq:canonical}, the set
$\B (\Sc_{q}) =\big\{\{\phi_{\gamma\nu}^g\}\ |\ (\gamma,g,\nu) \in \Xi \big\}$ forms an $\A$-basis of $\Sc_{q}$, which is called {\it the canonical basis}.
We summarize this as follows.

\begin{thm} \label{thm:CB-Sjj}
There exists a canonical basis $\B (\Sc_{q})=\big\{\{\phi_{\gamma\nu}^g\}\ |\ (\gamma,g,\nu) \in \Xi \big\}$ for $\Sc_{q}$,
which is characterized by the properties \eqref{barA}--\eqref{eq:canonical}.
\end{thm}

\begin{prop}
The commuting actions of $\Sc_q$ and $\HH$ on $\T_\ff$ are compatible with the bar maps, that is,
 \[
 \overline{\eta\cdot v \cdot h}=\overline{\eta} \cdot \overline{v} \cdot \overline{h},
 \]
 for all $\eta\in \Sc_q$, $v\in \T_\ff$ and $h\in \HH$.
\end{prop}

\begin{proof}
We already knew by \eqref{baranti} that
\begin{equation}
 \label{eq:2bar}
 \overline{v \cdot h}= \overline{v} \cdot \overline{h}.
\end{equation}
It remains to verify that $\overline{\eta\cdot v}=\overline{\eta} \cdot \overline{v}$, for all $\eta\in \Sc_q$, $v\in \T_\ff$. To that end, it suffices to check $\overline{f(x_\nu h)} =\bar{f} (\overline{x_\nu h})$, for any $f \in \Hom_{\HH}(x_\nu \HH, x_\gamma \HH) \subset \Sc_{q}$ and $h\in \HH$.
Indeed, by \eqref{eq:barS}, \eqref{eq:2bar} and the fact that $x_\nu$ is bar invariant, we have
\begin{align*}
\overline{f(x_\nu h)} =\overline{f(x_\nu) h} =\overline{f(x_\nu)} \bar{h}
=\overline{f}(x_\nu) \bar{h} =\overline{f}(x_\nu \bar{h})  =\overline{f}( \overline{x_\nu h}).
\end{align*}
The proposition is proved.
\end{proof}

\begin{rem}
  \label{rem:algebra:ABC}
For classical types, the $q$-Schur algebras and $q$-Schur dualities have been constructed in earlier works  \cite{DJ89, Gr97} (also see \cite{BW18}). The uniform formulation of $q$-Schur algebras of arbitrary finite type in this paper starts with a $W$-invariant subset $X_\ff$ of weights. For some distinguished choices of $X_\ff$ in classical types as specified in \S\ref{subsec:Pf}, we recover the earlier constructions {\em loc. cit.}; for example, in type A, our $q$-Schur algebra specializes to the familiar one often denoted by $\Sc_q(n,d)$. The relation between different versions of standard bases in $q$-Schur algebras of type A/B was formulated in \cite{DD91, Gr97}.
An algebraic construction of canonical bases of $q$-Schur algebras in type A was given in \cite{Du92}.
\end{rem}

\begin{rem}
In \cite[Example~ 1.4]{Du94}, Du introduced a particular $q$-Schur algebra of arbitrary type, and constructed its canonical basis. His construction corresponds to the special case of ours by choosing $X_\ff$ to consist of the union of $W$-orbits $W/W_J$, one for each subset $J \subseteq \{1,2,\ldots,d\}$.
\end{rem}

\section{A geometric setting for $q$-Schur algebras}
  \label{sec:geom}

In this section, we provide geometric realizations of the $q$-Schur algebras, the $q$-Schur dualities, and canonical bases.

\subsection{Convolution algebras}

Let $\F_{\q}$ be a finite field of $\q$ elements of characteristic $>3$. Let $\bf G$ be a connected algebraic group
defined over $\F_\q$. Assume $G={\bf G}(\F_\q)$ admits a split maximal torus and Borel subgroup, denoted by $T\subset B$. Let $W$ be the Weyl group of $G$. 
Associated to each subset $J\subset \{1,\ldots, d\}$, we have a standard parabolic subgroup $P_J$ which contains $B$ as a subgroup. In particular, $P_\emptyset =B$.
For $\gamma \in \Lambda_{\ff}$, recalling $J_\gamma$ from \eqref{eq:PP}, we denote $P_\gamma =P_{J_\gamma}$.

Recalling $\Lambda_\ff$ from \eqref{Ln}, we consider the following sets:
\begin{align*}
\sF &=\bigsqcup_{\gamma \in \Lambda_{\ff}} G/P_{\gamma},
\qquad
\sB =G/B.
\end{align*}
Clearly  $G$ acts on $\sF$ and $\sB$. Let $G$ act  diagonally on $\sF\times \sF$, $\sF\times \sB$ and $\sB\times \sB$, respectively.

Recall $X_{\ff}$ from \eqref{eq:Pf} and $\Xi$ from \eqref{eq:Xi}. Note $X_{\ff} =\cup_{\gamma \in \Lambda_{\ff}} X_\gamma$, where $X_\gamma =\{\mi \in X_{\ff}~|~\mi \sim \bu_\gamma\}$.
\begin{lem}
  \label{lem:bijectionorbits}
We have the following natural bijections:
\begin{equation}
\label{bijections}
G \backslash (\sF\times \sF) \longleftrightarrow \Xi, \quad
G \backslash (\sF\times \sB) \longleftrightarrow X_{\ff}, \quad
G \backslash (\sB\times \sB) \longleftrightarrow W.
\end{equation}
\end{lem}
(We shall denote these $G$-orbits by $\Ob_{\xi}, \Ob_\mi$, and $\Ob_w$, respectively.)

\begin{proof}
The bijection $G \backslash (\sB\times \sB) \longleftrightarrow W$ is standard. The bijection $G \backslash (\sF\times \sB) \longleftrightarrow X_{\ff}$ follows from composing the bijections $G \backslash (G/P_{\gamma}\times \sB) \longleftrightarrow \D_{\gamma} \longleftrightarrow X_\gamma$, for all $\gamma \in \Lambda_{\ff}$. By definition of $\Xi$ from \eqref{eq:Xi}, we have $\Xi \cong \sqcup_{\gamma, \nu \in \Lambda_{\ff}} \D_{\gamma\nu}$. Now the bijection $G \backslash (\sF\times \sF) \longleftrightarrow \Xi$ follows from the bijections $G \backslash (G/P_{\gamma}\times G/P_{\nu}) \longleftrightarrow \D_{\gamma\nu}$, for all $\gamma, \nu \in \Lambda_{\ff}$.
\end{proof}

We define
$$
\Scg =\Scg (X_\ff) := \A_{G}(\sF\times \sF), \qquad \T_{\ff}' :=\A_{G}(\sF\times \sB), \qquad \HH' :=\A_{G}(\sB\times \sB)
$$
to be the space of $G$-invariant $\A$-valued functions on $\sF \times \sF$, $\sF \times \sB$, and $\sB \times \sB$ respectively.
(Note by Lemma~\ref{lem:bijectionorbits} that the parametrizations of the $G$-orbits
are independent of the finite fields $\mathbb F_\q$.)
For $\xi \in \Xi$ (and $\mi\in X_\ff$, $w\in W$, respectively), we denote by $\chi_\xi$ (and $v'_\mi$, $H'_w$, respectively) the characteristic function of the orbit $\Ob_\xi$ (and $\Ob_\mi$, $\Ob_w$, respectively).
Then $\Scg$ is a free $\A$-module with a basis $\{\chi_\xi \mid \xi \in \Xi\}$. Similarly, $\T_{\ff}'$ and $\HH'$ are
free $\A$-modules with bases parameterized by $X_{\ff}$ and $W$, respectively.

We define a convolution product $*$ on $\Scg$ as follows.
For a triple $(\xi,\xi',\xi'')$ in $\Xi \times \Xi \times \Xi$,
we fix $(f_1, f_2) \in \Ob_{\xi''}$, and  let $\kappa_{\xi,\xi',\xi'';\q}$ be the number of $f\in \sF$ such that $(f_1, f) \in \Ob_{\xi}$ and $(f, f_2) \in \Ob_{\xi'}$.
A well-known property of the Iwahori-Hecke algebra implies that there exists a polynomial $\kappa_{\xi,\xi',\xi''} \in \Z[q^2]$ such that
$\kappa_{\xi,\xi',\xi'';\q} =\kappa_{\xi,\xi',\xi''}|_{q^{-2}=\q}$ for all prime powers $\q=p^r$ with primes $p>3$. We define the convolution product  on $\Scg$ by letting
\begin{equation}  \label{eq:conv}
\chi_\xi * \chi_{\xi'} =\sum_{\xi''} \kappa_{\xi,\xi',\xi''} \chi_{\xi''}.
\end{equation}
Equipped with the convolution product, the $\A$-module $\Scg$ becomes an associative $\A$-algebra.

According to Iwahori (cf. \cite[\S67]{CR87}), an analogous convolution product gives us an $\A$-algebra structure on
$\HH'$, which is identified with the Iwahori-Hecke algebra $\HH$ associated to $W$.

\subsection{A geometric setting for $q$-Schur algebra and duality}

A convolution product analogous to \eqref{eq:conv} for $\Scg$ by regarding  $(\xi,\xi',\xi'') \in \Xi \times X_{\ff} \times X_{\ff}$ gives us
a left $\Scg$-action on $\T_{\ff}'$; a suitably modified convolution gives us a right $\HH'$-action on $\T_{\ff}'$. These two actions commute and hence we have obtained an $(\Scg, \HH')$-bimodule structure on $\T_{\ff}'$.

Recall $\phi_{\gamma\nu}^g\in \Sc_{q}$ from \eqref{eq:phig}.
\begin{thm}
 \label{thm:same duality}
A geometric interpretation of the $q$-Schur duality \eqref{eq:duality} is provided by the following commutative diagram:
\begin{eqnarray}
   \label{CD}
\begin{array}{ccccc}
  \Scg &\circlearrowright
 & \T_{\ff}' & \circlearrowleft  \; & \HH'
  \\
  \downarrow\;\; & & \downarrow & & \parallel
 \\
\Sc_{q}     & \circlearrowright
& \T_{\ff} &  \circlearrowleft  \; & \HH
\end{array}
 \end{eqnarray}
Here the identifications are given by
\begin{align*}
 \chi_\xi & \mapsto  \phi^g_{\gamma\nu}, \; \quad\text{ for }\xi =(\gamma, g, \nu) \in \Xi,
\\
v_\mi' & \mapsto q^{-\ell(\sigma)} v_\mi, \quad \text{ for } \mi=\bu_\nu \sigma \text{ with } \sigma \in \D_\nu,
\\
H_w' &\mapsto q^{-\ell(w)} H_w, \quad \text{ for } w \in W.
\end{align*}
 \end{thm}
(The $q$-powers in the identifications above are resulted from our different conventions for various basis elements involved above.)

\begin{proof}
We shall verify that the map $\Theta: \Scg \longrightarrow \Sc_q$, $\chi_\xi\mapsto \phi^g_{\gamma\nu}$, is an algebra isomorphism.
To that end, we note that
\[
\Sc_q'|_{q^{-2}=\q} =\text{End}_G \left( \bigoplus_{\gamma \in \Lambda_\ff} \Ind_{P_\gamma}^G \Z \right )^{\text{op}}.
\]
This follows from the standard results on permutation modules; cf. \cite{CR87} or \cite[Exercise~13.3]{DDPW}.

The argument for verifying $\Theta$ is an algebra isomorphism is a verbatim repetition of the proof of \cite[Theorem 13.15]{DDPW} (though it was assumed in the type A setting therein). The key ingredients used in the argument include the first bijection in Lemma~\ref{lem:bijectionorbits}, Iwahori's geometric realization of Hecke algebra (i.e. $\HH'=\HH$), the Bruhat decomposition (and/or basics of $BN$-pairs). We refer to {\em loc. cit.} for the details.

Recall by assumption the $W$-invariant set $X_\ff$ contains a regular orbit.  The compatibility of $\T_\ff' \rightarrow \T_\ff$ with the left algebra actions can be viewed as some (simper) variation of the above algebra isomorphism $\Theta$. (On the geometric level, we have $\sB \subset \sF$ and $\sF\times \sB \subset \sF\times \sF$.)

The compatibility of the right Hecke algebra actions is standard and will be skipped.
Alternatively, as the $q$-Schur algebra contains the Hecke algebra as a subalgebra (by our assumption on $X_\ff$), the compatibility of the right algebra actions can be viewed as a (simpler) variant of the compatibility of the Schur algebra actions.
\end{proof}

\subsection{A canonical basis}
\label{subset:CBonS}

For $\xi =(\gamma, g, \nu) \in \Xi$, we denote by $d(\xi)$ the dimension of the $G$-orbit $\Ob_\xi$. Note that the dimension $d(\xi^{\tiny \Delta})$ of the $G$-orbit $\Ob_{\xi^{\Delta}}$, where $\xi^{\Delta} =(\gamma,\id,\gamma)$, is simply the dimension of $G/P_\gamma$.
Define
\[
[\xi] =q^{d(\xi)-d(\xi^{\tiny \Delta})} \chi_\xi.
\]
Then $\{[\xi] \mid \xi \in \Xi\}$ forms a basis for $\Scg$ (called a standard basis).

\begin{lem}
\label{lem:d=d}
Retain the notations above. Then we have
\begin{equation}  \label{eq:dim}
d(\xi)- d(\xi^{\tiny \Delta}) = \ell(g_{\gamma\nu}^+)-\ell(w_\circ^\nu).
\end{equation}
\end{lem}

\begin{proof}
By \cite[Lemma A.2]{Du92} (stated there for type A) and its proof, which is indeed valid for any finite type, we have
\begin{equation}  \label{eq:dim}
d(\xi)  = \ell(g_{\gamma\nu}^+)+\ell(w_\circ)-\ell(w_\circ^\la)-\ell(w_\circ^\nu),
\end{equation}
where $w_\circ$ is the unique longest element in $W$.
Applying \eqref{eq:dim} again to the orbit $\Ob_{\xi^{\Delta}}$ gives us
\begin{equation} \label{eq:dimD}
d(\xi^{\Delta})  = \ell(w_\circ)-\ell(w_\circ^\la).
\end{equation}
Now the lemma follows from \eqref{eq:dim}--\eqref{eq:dimD}.
\end{proof}
Recall from \eqref{def:dA} and Lemma~\ref{lem:d=d} that
$[\phi_{\gamma\nu}^g]=q^{\ell(g_{\gamma\nu}^+)-\ell(w_\circ^\nu)}\phi_{\gamma\nu}^g.$
It follows by Theorem~\ref{thm:same duality} that
\begin{equation}  \label{eq:samestd}
\Theta ([\xi])=[\phi^g_{\gamma,\nu}].
\end{equation}

Let $\text{IC}_\xi$, for $\xi\in \Xi$, be  the shifted intersection complex associated with the closure of  the orbit $\Ob_\xi$
such that the restriction of $\text{IC}_\xi$ to $\Ob_\xi$ is the constant sheaf on  $\Ob_\xi$.
Since $\text{IC}_\xi$ is $G$-equivariant, the stalks of the $i$-th cohomology sheaf of $\text{IC}_\xi$  at different points in $\Ob_{\xi'}$ (for $\xi' \in \Xi$) are isomorphic.
Let $\mathscr H^i_{\Ob_{\xi'}} (\text{IC}_\xi)$ denote the stalk of the $i$-th cohomology group of $\text{IC}_\xi$ at any point in $\Ob_{\xi'}$.
We  set
\begin{align}
 \label{eq:A}
\begin{split}
P_{\xi', \xi} &=\sum_{i\in \Z} \dim \mathscr H^i_{\Ob_{\xi'}} (\text{IC}_\xi) \; q^{-i+d(\xi) -d(\xi')},
 \\
\{ \xi\} &= \sum_{\xi'\leq \xi} P_{\xi', \xi} [\xi'],
\end{split}
\end{align}
where the partial order ``$<$'' on $\Xi$ is a standard Bruhat order, or equivalently, the orbit closure order. Precisely, for $\xi'=(\gamma', g', \nu')$ and $\xi=(\gamma, g, \nu)$,
$$\xi'< \xi \quad \Leftrightarrow \quad \gamma'=\gamma,\nu'=\nu, g'<g.$$

By the properties of intersection complexes, we have
\begin{equation}
 \label{eq:v-}
P_{\xi, \xi} =1, \qquad P_{\xi', \xi} \in q \N [q] \; \; \mbox{ for }  \; \xi' < \xi.
\end{equation}
As in \cite[1.4]{BLM90}, we have an anti-linear bar involution $\bar\ : \Scg \to \Scg$ such that
\[
\overline{\{\xi\}} =\{\xi\}, \quad \forall \xi\in \Xi.
\]
In particular, we have
\[
\overline{[\xi]} =\sum_{\xi'\leq \xi} c_{\xi', \xi} [\xi'], \quad \mbox{where}\; c_{\xi, \xi} =1, \; c_{\xi', \xi}\in \Z[q,q^{-1}].
\]
Then $\B (\Scg) := \{ \{\xi\} \mid \xi\in \Xi\}$ forms an $\A$-basis for $\Scg$, called a {\em canonical basis}.
There is a similar bar involution on $\T_\ff'$, and the bar maps are compatible with the commuting actions of $(\Scg, \HH')$ on $\T_\ff'$.

\begin{prop}
  \label{prop:CB=CB}
The isomorphism $\Theta :\Scg \rightarrow \Sc_{q}$ matches the canonical bases $\B (\Scg)$ and $\B (\Sc_{q})$.
\end{prop}

\begin{proof}
This follows from the identification of standard bases under $\Theta$, and the uniqueness of the canonical basis via the bar invariance and that $\{\xi\} \in [\xi] +\sum_{\xi'} q\Z[q] [\xi']$. Note the uniqueness does not require any partial ordering condition on $\xi'$.
\end{proof}

Recall the canonical bases $\B (\Sc_{q})$ and $\B (\T_{\ff})$ for the $\A$-algebra $\Sc_{q}$ and its module $\T_{\ff}$, respectively.
\begin{thm} [Positivity]
For any $c \in \B (\T_{\ff})$ and $b, b' \in \B (\Sc_{q})$, we write
\[
b b' =\sum_{b''\in \B (\Sc_{q}) } m_{b,b'}^{b''} b'', \qquad
b\cdot c =\sum_{c''\in \B (\T_{\ff}) } t_{b,c}^{c''} c'', \quad \text{ for } m_{b,b'}^{b''}, t_{b,c}^{c''}  \in \A.
\]
Then we must have
$m_{b,b'}^{b''}  \in \N[q,q^{-1}]$, and
$t_{b,c}^{c''}  \in \N[q,q^{-1}]. $
\end{thm}

\begin{proof}
This follows from the geometric interpretation of these canonical bases and their multiplication/action in terms of perverse sheaves and their convolution products.
\end{proof}

\begin{rem}
Two $W$-invariant subsets $X_\ff \subset \tilde X_\ff$ give rise to two $q$-Schur algebras $\Sc_q(X_\ff) \subset \Sc_q(\tilde X_\ff)$. By the geometric interpretation, their canonical bases are compatible, i.e., $\B (\Sc_{q}(X_\ff)) \subset \B (\Sc_{q}(\tilde X_\ff))$. Taking the limit, we obtain a canonical basis for the infinite-rank $q$-Schur algebras $\Sc_q(X) =\End_\HH (\T)$.
\end{rem}

\subsection{Earlier geometric $q$-Schur algebras of classical types}

We refer to Remark~\ref{rem:algebra:ABC} for the algebraic constructions of $q$-Schur algebras in the literature.

For classical types, the $q$-Schur algebras have been constructed geometrically in earlier works  (see \cite{BLM90} for type A, see \cite{BKLW18, FL15} and \cite[Appendix~A]{FLLLW} for type BCD), where $\sF$ is formulated naturally as the variety of ``$n$-step flags". As each classical type forms an infinite family (depending on the rank $d$), one further establishes some remarkable stability properties of the $q$-Schur algebras as $d$ goes to infinity and to realize a suitable ``quantum group" as a stablization limit of the family of $q$-Schur algebras.

Our geometric formulation of $q$-Schur algebras of arbitrary finite types is uniform, again starting with a $W$-invariant finite set $X_\ff$ of integral weights. For the particular choices of $X_\ff$ in classical type as specified in \S\ref{subsec:Pf}, we recover the earlier constructions {\em loc. cit.}.

\subsection{Connections to BGG category $\OO$}

Our general constructions of $q$-Schur algebras and dualities have been further motivated by their connections to the BGG Category $\OO$; see for example \cite{Br03, CLW15} for type A, and \cite{BW18, Bao17} and references therein for the connections in classical type. To this end, the module $\T_{\ff}$ over the Hecke algebra at $q=1$ can be identified as the Grothendieck group of a truncated category $\OO$, and its decomposition into a direct sum of permutation modules is interpreted as a decomposition into blocks. A generating set for the $q$-Schur algebra is expected to be provided by translation functors. In such a formulation, the standard basis for $\T_{\ff}$ corresponds to the Verma modules while the canonical basis for $\T_{\ff}$ corresponds to the tilting modules (as a slight reformulation of Kazhdan-Lusztig theory).

We shall make such a connection to BGG category $\OO$ precise in type $G_2$ in the next section.

\section{Category $\OO$ and the $q$-Schur algebras of type $G_2$}
   \label{sec:G2}

In this section, specializing to the type $G_2$, we present more precise results and formulate connections to the BGG category $\OO$.

\subsection{The basics for $G_2$}

Let $\mathfrak{g}$ be the Lie algebra of type $G_2$, and let $\mathfrak{h} \subset \mathfrak b$ be a Cartan and a Borel subalgebra of $\mathfrak{g}$.

We let $\C^3$ be equipped with a standard orthonormal basis $\{\del_1, \del_2, \del_3\}$.
We take $\h^*$ to be the orthogonal complement to $\del_1+\del_2+\del_3$, i.e.,
\begin{align}
\h^* &=\{a_1 \del_1 +a_2 \del_2 +a_3 \del_3 \mid a_1 +a_2 +a_3 =0\}.
\end{align}
We can take a simple system (of type $G_2$) $\Pi =\{\alpha_1, \alpha_2\}$,
where
\[
\al_1 = \del_1 -\del_2,
\qquad
\al_2 = -2 \del_1 +\del_2 +\del_3.
\]
The fundamental weights with respect to $\Pi$ are
\begin{align}
\label{fund}
 \omega_1 = -\del_2 +\del_3,
 \qquad
 \omega_2 =-\del_1-\del_2 +2\del_3.
\end{align}
The weight lattice is
\begin{align}
X 
&=\{a \del_1 +b\del_2 +c\del_3 \mid a+b+c=0, a,b,c\in \Z \}.
\end{align}
The set of anti-dominant integral weights is given by
\[
X^- =\{a \del_1 +b\del_2 +c\del_3 \in X \mid 0\leq a\leq b\}.
\]
To make the notation more concise, we use the following convention
\begin{align*}
X&= \{(a,b,c) \in \Z^3 \mid a+b+c=0\},
\\
X^-&= \{(a,b,c) \in \Z^3 \mid a+b+c=0, 0\leq a\leq b \},
\end{align*}
where $(a,b,c)$ means the weight $a \del_1 +b\del_2 +c\del_3$.

The Weyl group $W$ (of type $G_2$) is generated by $s_1$ and $s_2$ with relations
$s_1^2=s_2^2=\id, \, (s_1s_2)^3=(s_2s_1)^3.$
There is a natural action of $W$ on $\C^3$ given by
\begin{align}
\label{Waction}
  \begin{split}
s_1: & \; \del_1 \mapsto \del_2, \quad \del_2 \mapsto \del_1,  \quad \del_3 \mapsto \del_3,
\\
s_2: & \; \del_1 \mapsto -\del_1, \quad \del_2 \mapsto -\del_3,  \quad \del_3 \mapsto -\del_2.
  \end{split}
\end{align}
The action of $W$ on $\h^*$ is given by the restriction to $\h^*$ of the above action on $\C^3$.
The Bruhat graph of $W$ is defined as follows:
\begin{equation}
 \label{Bruhat}
\xymatrix@R=0em@C=1em{
& s_1  \ar@{->}[r] \ar@{->}[rdd] & s_1s_2 \ar@{->}[r] \ar@{->}[rdd] & s_1s_2s_1 \ar@{->}[r] \ar@{->}[rdd] & s_1s_2 s_1s_2\ar@{->}[r] \ar@{->}[rdd]  &   s_1s_2s_1s_2s_1 \ar@{->}[dr] &
  \\
\id \ar@{->}[ur]  \ar@{->}[dr] &&&&&& s_1s_2s_1s_2s_1s_2. \\
& s_2 \ar@{->}[r] \ar@{->}[ruu]  & s_2s_1 \ar@{->}[r] \ar@{->}[ruu]&s_2s_1s_2 \ar@{->}[r] \ar@{->}[ruu]&s_2s_1s_2s_1 \ar@{->}[r] \ar@{->}[ruu]&s_2s_1s_2s_1s_2 \ar@{->}[ru] &
}
\end{equation}

The Hecke algebra $\HH$ (of type $G_2$) is an $\A$-algebra generated by $H_1$ and $H_2$ with relations
\[
(H_1-q^{-1})(H_1+q)=(H_2-q^{-1})(H_2+q)=0; \quad (H_1H_2)^3=(H_2H_1)^3.
\]

\subsection{The ``$q$-Grothendieck groups''  of type $G_2$}

The linkage classes (i.e., the $W$-orbits) in $X$ are induced from the following equivalence relation $\sim$ on $X$: for $(a_1,b_1,c_1),(a_2,b_2,c_2)\in X$,
\begin{equation}\label{linkage}
(a_1,b_1,c_1)\sim(a_2,b_2,c_2)\quad\Leftrightarrow\quad\{|a_1|,|b_1|,|c_1|\}=\{|a_2|,|b_2|,|c_2|\}.
\end{equation}

Given $n \in \N$, set
\begin{align}
\label{eq:Qng2}
X_{n} &= \{ (a,b,c) \in X \mid (a,b,c)\sim(a',b',c'), \text{ for some } 0\leq a'\leq b'\leq n \},
\\
X_{n}^- &=X_{n}\cap X^-.
\notag
\end{align}

In each linkage class in $X$, there exists a unique anti-dominant element $(a,b,c)\in X^-$.
We introduce symbols $\varepsilon_a$, for $a\in \Z$, and call $\varepsilon_a+\varepsilon_b$ the \emph{$\varepsilon$-weight} of all elements in the class of $(a,b,c) \in X^-$.
Denote the set of $\varepsilon$-weights, for $n\in \N$, by
\begin{align}
  \label{eq:Lambdag2}
\Lambda &=\{\varepsilon_a+\varepsilon_b~|~a,b\in \Z,\; 0\leq a\leq b\},
\\
\Lambda_n &=\{\varepsilon_a+\varepsilon_b~|~a,b\in \Z, \; 0\leq a\leq b\leq n\},
\label{Lng2}
\\
\bu_\gamma & =\mbox{ the unique anti-dominant element in } X \mbox{ of $\varepsilon$-weight } \gamma \in \Lambda.
\label{eq:ibug2}
\end{align}
In other words, a linkage class in $X$ is the same as a set of elements in $X$ having the same $\varepsilon$-weight, and  there is a bijection  
\begin{equation}
  \label{eq:XnLn}
X_n^- \leftrightarrow \Lambda_n,
\text{ where } (a,b,c)\mapsto \varepsilon_a+\varepsilon_b
\text{ with inverse given by } \gamma \mapsto \bu_\gamma.
\end{equation}

The notations $\Lambda$ and $\Lambda_n$ in \eqref{eq:Lambdag2}--\eqref{Lng2} look seemingly different from
those in \eqref{eq:Lambda}--\eqref{Ln}, but they are essentially the same.

As in \eqref{eq:Tn} and \eqref{ZformT},
we introduce the following free $\A$-modules
\begin{equation}\label{eq:Tng2}
\T =\bigoplus_{\mi\in X} \A v_\mi,
\qquad
 \T_n =\bigoplus_{\mi\in X_n} \A v_\mi,
\end{equation}
and their $\Z$-forms
\begin{equation}\label{ZformTg2}
\T^1=\bigoplus_{\mi\in X}\Z v_\mi, \quad \quad \T_n^1=\bigoplus_{\mi\in X_n}\Z v_\mi.
\end{equation}

\begin{rem}
 \label{rem:dimTn}
The $W$-orbits on $X_n$ are classified as follows: ${n \choose 2}$ orbits are regular, $n$ orbits can be identified with $W/\langle s_1\rangle$, $n$ orbits can be identified with $W/\langle s_2\rangle$, and in addition, there is one singleton orbit.
This implies that
\begin{equation}  \label{dim:Tn}
\mbox{rank}_\A (\T_n) = |X_n| = 12 \cdot {n \choose 2} +6\cdot n +6 \cdot n + 1 =6n^2 +6n+1.
\end{equation}
\end{rem}

The natural right action of $W$  on $\C^3$ induces a right action of $W$ on $\T$ (and $\T_n$) by permuting the indices of the standard basis elements $v_\mi$.
Then the right action of the Hecke algebra $\HH$ on $\T$ (and $\T_n$) is given as follows:
\begin{align}
v_{(a,b,c)}H_1&=
\left\{\begin{array}{ll}
q^{-1}v_{(a,b,c)}, & a=b;\\
v_{(b,a,c)}, & a<b;\\
v_{(b,a,c)}+(q^{-1}-q)v_{(a,b,c)}, & a>b ;
\end{array}\right.\label{H1action}
\\
v_{(a,b,c)}H_2&=\left\{\begin{array}{ll} q^{-1}v_{(a,b,c)}, & a=0;\\
v_{(-a,-c,-b)}, & a>0;\\
v_{(-a,-c,-b)}+(q^{-1}-q)v_{(a,b,c)}, & a<0.
\end{array}\right.\label{H2action}
\end{align}

\subsection{The $q$-Schur algebra of type $G_2$}

The $q$-Schur algebra of type $G_2$ is defined to be $\Sc_q(n) = \End_\HH (\T_n)$. Its specialization at $q=1$ is the Schur algebra (of type $G_2$) over $\Z$, saying $\Sc(n) = \End_W (\T_n^1)$. Set $\Sc_q(n)_\Q =\Q(q)\otimes_\A \Sc_q(n)$.

\begin{prop}
The dimension of the algebra $\Sc_q(n)_\Q$ is 
$3n^4+6n^3+6n^2+3n+1$.
\end{prop}

\begin{proof}
For the sake of simplifying notations, let us work with the $q=1$ specialization over the complex field $\C$ in the proof. (The same argument works over an algebraic closure of $\Q(q)$.) Denote $\Sc(n)_\C =\C \otimes_\Z \Sc(n)$.

First assume $n\ge 2$. The double centralizer property gives us a multiplicity-free decomposition of $(\Sc(n)_\C, \C W)$-modules
\[
\C\otimes_\Z \T^1_n \cong \sum_{\varrho \in \text{Irr} W} L_\varrho \otimes \varrho,
\]
where $\text{Irr} W$ denotes the set of inequivalent irreducible $\C W$-modules and $\{L_\varrho\}_{\varrho \in \text{Irr} W}$ denotes the set of inequivalent irreducible $\Sc (n)_\C$-modules.

Note $W$ has 6 irreducible modules over $\C$. The 2 irreducibles of dimension 2 are denoted by $\underline{2}_1, \underline{2}_2$ while the remaining 4 irreducibles of  dimension 1 are denoted by $\underline{1}_1, \underline{1}_2, \underline{1}_3, \underline{1}_4$, with $\underline{1}_1$ being trivial and $\underline{1}_2$ being the sign module. As $W$-modules, we have
\begin{align*}
\C W &\cong \underline{2}_1^{\oplus 2} \oplus \underline{2}_2^{\oplus 2} \oplus \underline{1}_1 \oplus \underline{1}_2 \oplus \underline{1}_3 \oplus \underline{1}_4,
\\
\Ind^{W}_{\langle s_1\rangle} (\text{\bf Triv}) &\cong \underline{2}_1 \oplus \underline{2}_2 \oplus \underline{1}_1 \oplus \underline{1}_3,
\qquad
\Ind^{W}_{\langle s_2\rangle} (\text{\bf Triv}) \cong  \underline{2}_1 \oplus \underline{2}_2 \oplus \underline{1}_1 \oplus \underline{1}_4,
\end{align*}
where ${\bf Triv}$ denotes the trivial module.
From these we count the multiplicities of the simple modules $\underline{2}_1, \underline{2}_2$, $\underline{1}_1, \underline{1}_2, \underline{1}_3, \underline{1}_4$ in the $W$-module $\T^1_n$ to be
$2 {n \choose 2} +n+n, \;
2 {n \choose 2} +n+n, \;
{n \choose 2} +n +n +1, \;
{n \choose 2}, \;
{n \choose 2} +n, \;
{n \choose 2} +n.
$
Therefore the dimensions of the corresponding simple $\Sc(n)_\C$-modules are
\begin{equation}  \label{dim:Sirrep}
2{n+1 \choose 2}, \quad
2{n+1 \choose 2}, \quad
{n+2 \choose 2}, \quad
{n \choose 2}, \quad
{n+1 \choose 2}, \quad
{n+1 \choose 2},
\end{equation}
respectively.
The dimension of $\Sc_q(n)_\Q$ is equal to the sum of the squares of the numbers in \eqref{dim:Sirrep} by the double centralizer property (Theorem~\ref{thm:double}), which can be rewritten as in the proposition.

In case $n=1$, the regular representation $\C W$ will not appear, which corresponds to ${1 \choose 2}=0$ in the above calculation; hence the  dimension formula above for $\Sc_q(n)_\Q$ continues to hold.
\end{proof}

\begin{rem}
The transversal $X_n^{(2)}$ to the $W$-orbits in $X_n \times X_n$ consists of elements $(\mi,\mj)$ with $\mi=(i_1,i_2,i_3)$ and $\mj=(j_1,j_2,j_3)$
satisfying the following conditions:
\begin{itemize}
\item[(i)] $0\leq j_1\leq j_2$;

\item[(ii)] If $j_1=0$ then $i_1\geq0$;

\item[(iii)] If $j_1=j_2$ then $i_1\leq i_2$.
\end{itemize}
The dimension of $\Sc_q(n)_\Q$ can also be determined by computing the size of $X_n^{(2)}$.
\end{rem}


\subsection{The Grothendieck group $[\OO_n]$ and $\T_n^1$}
 \label{transfunc}

It is well known that the Lie algebra $\g$ of type $G_2$ can be embedded into the orthogonal Lie algebra
$\mathfrak{so}(7)$, and the natural $\mathfrak{so}(7)$-module $\underline{\bf 7}$ restricts to be the minimal faithful representation of $\mathfrak{g}$, which can be identified with a fundamental weight $\mathfrak{g}$-module of highest weight $\omega_1$ (recall $\omega_1=\delta_3-\delta_2$ in \eqref{fund}).
As a $\mathfrak{g}$-module, the module $\underline{\bf 7}$ has weights
$\{0,\delta_i-\delta_j \; (1\le i\neq j \le 3)\}.$

Recall the BGG category $\mathcal O$ is the category of all finitely generated, integral-weight, $\mathfrak{g}$-modules $M$ which are locally finite over $\mathfrak b$.
Let $[\OO]$ denote its Grothendieck group.
The action of the center of the enveloping algebra of $\mathfrak g$ on a Verma module $M(\mi)$ of highest weight $\mi-\rho\in X$ yields a central character $\chi_\mi$, where
\begin{equation*}
\rho =-\del_1 -2\del_2+3\del_3\in X.
\end{equation*}
is the half sum of positive roots.
We have the following equivalences
\[
\chi_\mi=\chi_\mj \quad\Leftrightarrow \quad \mi \sim \mj
\quad\Leftrightarrow \quad \mi, \mj \text{ have the same $\varepsilon$-weight}
\]
(see \eqref{linkage} for notations).
Let $\OO_\gamma$ denote the full subcategory of $\OO$ consisting of the modules all of whose composition factors have central character $\chi_\la$, where $\gamma$ is the $\varepsilon$-weight of $\la$. Recall the set of $\varepsilon$-weights $\Lambda$ in \eqref{eq:Lambdag2}.
We have the block decomposition
\[
\OO=\bigoplus_{\gamma\in\Lambda}\OO_\gamma.
\]
For $n\in \N$, we also consider the full subcategory $\OO_n$ of $\OO$:
\[
\OO_n=\bigoplus_{\gamma\in\Lambda_n}\OO_\gamma.
\]
We let $\Pro_\gamma:\OO\rightarrow \OO_\gamma$ be the natural projection.
We define translation functors
\[
e_a, f_{a}, t: \OO\rightarrow\OO,
\qquad \text{ for } a\in\N,
\]
which are additive such that for each $\gamma\in\Lambda$ and $M\in\OO_\gamma$,
\begin{align}\label{trans}
\begin{split}
e_a(M)&:=\Pro_{\gamma-\varepsilon_{a+1}+\varepsilon_a}(M\otimes \underline{\bf 7}), \\
f_a(M)&:=\Pro_{\gamma+\varepsilon_{a+1}-\varepsilon_a}(M\otimes \underline{\bf 7}), \\
t(M)&:=\Pro_{\gamma}(M\otimes \underline{\bf 7}).
\end{split}
\end{align}

Recall $M(\mi)$ the Verma module of $\mathfrak{g}$ of highest weight $\mi-\rho\in X$. The following lemma is standard.
\begin{lem}\label{decomVerma}
For any $\mi\in X$, the $\mathfrak{g}$-module $M(\mi)\otimes \underline{\bf 7}$ admits a standard filtration whose subquotients isomorphic to $M(\mi+\mj)$, where $\mj$ runs over  $\{0, \delta_i-\delta_j(1\le i\neq j \le 3)\}$.
\end{lem}

Recall $\T^1$ and $\T_n^1$ in \eqref{ZformTg2}. We have natural $\Z$-linear isomorphisms
\begin{align}
  \label{OnTn}
 \begin{split}
\psi: [\OO]\longrightarrow \T^1, & \qquad \psi_n :[\OO_n]\longrightarrow \T_n^1,
\\
 [M(\mi)] & \mapsto v_{\mi}.
  \end{split}
\end{align}
Denote by $T(\mi)$, for $\mi \in X$, the tilting module of highest weight $\mi-\rho$.
The Kazhdan-Lusztig conjecture can be reformulated as follows: the isomorphism $\psi$ sends the tilting module $[T(\mi)]$ to the canonical basis $C_{\mi}$ specialized at $q=1$. We remark that such a reformulation of KL theory makes sense for any finite type.

\subsection{Identifying translation functors with $q$-Schur generators}

Recall the notations $\phi_{\gamma\nu}^g, [\phi_{\gamma\nu}^g]$ and $\{\phi_{\gamma\nu}^g\}$ in \eqref{eq:phig}, \eqref{def:dA} and \eqref{canobasis}, respectively.
We introduce the following elements $e_a,f_a,t$ in $\Sc_q(n)$, for $0\le a <n$:
\begin{align}
\label{eq:eftq}
\begin{split}
e_a&=\sum_{k=0}^{n}\{\phi_{\varepsilon_a+\varepsilon_{k},\varepsilon_{a+1}+\varepsilon_{k}}^\id\}=\sum_{k=0}^{n}[\phi_{\varepsilon_a+\varepsilon_{k},\varepsilon_{a+1}+\varepsilon_{k}}^\id],
\\
f_a&=\sum_{k=0}^{n}\{\phi_{\varepsilon_{a+1}+\varepsilon_{k},\varepsilon_{a}+\varepsilon_{k}}^\id\}=\sum_{k=0}^{n}[\phi_{\varepsilon_{a+1}+\varepsilon_{k},\varepsilon_{a}+\varepsilon_{k}}^\id],
\\
t& =\sum_{0<k+1<l}\{\phi_{\varepsilon_k+\varepsilon_l,\varepsilon_k+\varepsilon_l}^\id\}
+ \sum_{k=0}^{n-1}\{\phi_{\varepsilon_k+\varepsilon_{k+1},\varepsilon_k+\varepsilon_{k+1}}^{s_1}\}
\\
&= \sum_{0<k+1<l}[\phi_{\varepsilon_k+\varepsilon_l,\varepsilon_k+\varepsilon_l}^\id]
+ \left([\phi_{\varepsilon_0+\varepsilon_1,\varepsilon_0+\varepsilon_1}^{s_1}]+q^2[\phi_{\varepsilon_0+\varepsilon_1,\varepsilon_0+\varepsilon_1}^\id]\right)
\\
&\qquad\quad +\sum_{k=1}^{n-1}\left([\phi_{\varepsilon_k+\varepsilon_{k+1},\varepsilon_k+\varepsilon_{k+1}}^{s_1}]+q[\phi_{\varepsilon_k+\varepsilon_{k+1},\varepsilon_k+\varepsilon_{k+1}}^\id]\right).
\end{split}
\end{align}
The identification of the two formulas for $t$ above uses \eqref{eq:CB1} and the fact that the Kazhdan-Lusztig polynomials for type $G_2$ are
\begin{equation*}
P_{y,g}=
\left\{
\begin{array}{ll}
1, & \text{ if } y\le g;\\
0, &\mbox{otherwise}.
\end{array}
\right.
\end{equation*}

By definition, $t, e_a, f_a (0\leq a<n)$ are bar invariant.

\begin{prop}
Keep the identification $[\OO_n]\cong \T^1_n$ in \eqref{OnTn}. Then the following holds.
\begin{enumerate}
\item
The translation functors $e_a$, $f_a$ $(0\le a<n)$, and $t$ in \eqref{trans} acting on $[\OO_n]$ can be identified with the $q=1$ specialization of the elements of $\Sc_q(n)$ in  \eqref{eq:eftq} in the same notations acting on $\T_n$.
\item
The explicit formulas for the actions of $e_a$, $f_a$ $(0\le a<n)$ and $t$ in  \eqref{eq:eftq} on the standard basis for $\T_n$ are given by the formulas \eqref{eq:e}--\eqref{eq:t} in the Appendix~\ref{app:B}.
\end{enumerate}
\end{prop}

\begin{proof}
The proof is computational. One verifies (2) first by a direct computation. Then using Lemma \ref{decomVerma} one verifies that the translation functors $e_a$, $f_a$ and $t$ in \eqref{trans} acting on the Verma basis in $[\OO_n]$ are given by  the $q=1$ specialization of the formulas \eqref{eq:e}--\eqref{eq:t} in the Appendix~\ref{app:B}, proving (1). We skip the details.
\end{proof}

\begin{prop}\label{generators}
The $q$-Schur algebra $\Sc_q(n)_\Q$ is generated by $e_a, f_a (0\leq a<n)$ and $t$.
\end{prop}
\begin{proof}
The computational proof is given in Appendix~\ref{app:C}.
\end{proof}

\begin{rem}
  \label{rem:idemp}
The generators $e_a, f_a, t$ in \eqref{eq:eftq} are sums of canonical basis elements in $\Sc_q(n)$. Clearly the idempotents $\id_\gamma$ associated to each linkage class $\gamma$ belongs to the algebra $\Sc_q(n)$; note $1 =\sum_\gamma \id_\gamma$ is the identity. Then the (new) generators $e_a \id_\gamma, f_a \id_\gamma, t \id_\gamma$, if nonzero, are canonical basis elements; they are the simplest canonical basis elements in a suitable sense (like in classical types).
\end{rem}

\begin{rem}
  \label{rem:typeEF}
The translation functors and the generating sets for $q$-Schur algebras of types $F_4$ and $E_{6,7,8}$ remain to be worked out. It will be interesting to see if the phenomenon observed in Remark~\ref{rem:idemp} carries over.
\end{rem}

It seems likely that there are conceptual connections between the canonical basis of the $q$-Schur algebra and the graded indecomposable projective functors in a Koszul graded lift of the category $\OO$.

\appendix

\section{Explicit formulas in type $G_2$}
  \label{sec:app}

In this appendix, we consider the type $G_2$ case only. We collect here various formulas and computations needed in Section~\ref{sec:G2}.
\subsection{Formulas for the bar involution on $\T$}
  \label{app:A}

Recall $\overline{H_k} =H_k^{-1} =H_k+(q-q^{-1})$, for $k=1,2$. It follows from \eqref{H1action}--\eqref{H2action} that
\begin{align*}
v_{(a,b,c)}\overline{H_1} 
&=\left\{\begin{array}{ll} q v_{(a,b,c)}, & \mbox { if }  a=b,\\
v_{(b,a,c)}, & \mbox { if } a>b,\\
v_{(b,a,c)}-(q^{-1}-q)v_{(a,b,c)}, & \mbox { if } a<b;
\end{array}\right.
\\
v_{(a,b,c)}\overline{H_2} 
&=\left\{\begin{array}{ll} q v_{(a,b,c)}, & \mbox { if } a=0,\\
v_{(-a,-c,-b)}, & \mbox { if } a<0,\\
v_{(-a,-c,-b)}-(q^{-1}-q)v_{(a,b,c)}, & \mbox { if } a>0.
\end{array}\right.
\end{align*}

From the above formulas together with \eqref{H1action}--\eqref{H2action}, we can compute
the following formulas for the bar involution on every standard basis element $v_{\mi}$ of $\T$: for all $b>a>0$,
\begin{align*}
\overline{v_{(0,0,0)}}=&v_{(0,0,0)};
\\
\overline{v_{(0,a,-a)}}=&v_{(0,a,-a)};\\
\overline{v_{(a,0,-a)}}=&v_{(a,0,-a)}-(q^{-1}-q)v_{(0,a,-a)};\\
\overline{v_{(-a,a,0)}}=&v_{(-a,a,0)}-(q^{-1}-q)v_{(a,0,-a)}-(q^{-1}-q)qv_{(0,a,-a)};\\
\overline{v_{(a,-a,0)}}=&v_{(a,-a,0)}-(q^{-1}-q)v_{(-a,a,0)}-(q^{-1}-q)qv_{(a,0,-a)}-(q^{-1}-q)q^2v_{(0,a,-a)};\\
\overline{v_{(-a,0,a)}}=&v_{(-a,0,a)}-(q^{-1}-q)v_{(a,-a,0)}-
(q^{-1}-q)q v_{(-a,a,0)}-(q^{-1}-q)q^2v_{(a,0,-a)}\\
&-(q^{-1}-q)q^3v_{(0,a,-a)};\\
\overline{v_{(0,-a,a)}}=&v_{(0,-a,a)}-(q^{-1}-q)v_{(-a,0,a)}-(q^{-1}-q)qv_{(a,-a,0)}-
(q^{-1}-q)q^2v_{(-a,a,0)}\\&-(q^{-1}-q)q^3v_{(a,0,-a)}-(q^{-1}-q)q^4v_{(0,a,-a)};\\
\overline{v_{(a,a,-2a)}}=&v_{(a,a,-2a)};\\
\overline{v_{(-a,2a,-a)}}=&v_{(-a,2a,-a)}-(q^{-1}-q)v_{(a,a,-2a)};\\
\overline{v_{(2a,-a,-a)}}=&v_{(2a,-a,-a)}-(q^{-1}-q)v_{(-a,2a,-a)}-(q^{-1}-q)qv_{(a,a,-2a)};\\
\overline{v_{(-2a,a,a)}}=&v_{(-2a,a,a)}-(q^{-1}-q)v_{(2a,-a,-a)}-(q^{-1}-q)qv_{(-a,2a,-a)}-(q^{-1}-q)q^2v_{(a,a,-2a)};\\
\overline{v_{(a,-2a,a)}}=&v_{(a,-2a,a)}-(q^{-1}-q)v_{(-2a,a,a)}-
(q^{-1}-q)qv_{(2a,-a,-a)}\\
&-(q^{-1}-q)q^2v_{(-a,2a,-a)}-(q^{-1}-q)q^3v_{(a,a,-2a)};\\
\overline{v_{(-a,-a,2a)}}=&v_{(-a,-a,2a)}-(q^{-1}-q)v_{(a,-2a,a)}-(q^{-1}-q)qv_{(-2a,a,a)}-
(q^{-1}-q)q^2v_{(2a,-a,-a)}\\
&-(q^{-1}-q)q^3v_{(-a,2a,-a)}-(q^{-1}-q)q^4v_{(a,a,-2a)};\\
\overline{v_{(a,b,-b-a)}}=&v_{(a,b,-b-a)};\\
\overline{v_{(b,a,-b-a)}}=&v_{(b,a,-b-a)}-(q^{-1}-q)v_{(a,b,-b-a)};\\
\overline{v_{(-a,b+a,-b)}}=&v_{(-a,b+a,-b)}-(q^{-1}-q)v_{(a,b,-b-a)};\\
\overline{v_{(-b,b+a,-a)}}=&v_{(-b,b+a,-a)}-(q^{-1}-q)(v_{(b,a,-b-a)}
+v_{(-a,b+a,-b)})+(q^{-1}-q)^2v_{(a,b,-b-a)};\\
\overline{v_{(b+a,-a,-b)}}=&v_{(b+a,-a,-b)}-(q^{-1}-q)(v_{(b,a,-b-a)}
+v_{(-a,b+a,-b)})+(q^{-1}-q)^2v_{(a,b,-b-a)};\\
\overline{v_{(b+a,-b,-a)}}=&v_{(b+a,-b,-a)}-(q^{-1}-q)(v_{(-b,b+a,-a)}+v_{(b+a,-a,-b)})\\
&+(q^{-1}-q)^2(v_{(-a,b+a,-b)}+v_{(b,a,-b-a)})\\
&-(q^{-1}-q)(q^{-2}-1+q^2)v_{(a,b,-b-a)};\\
\overline{v_{(-b-a,b,a)}}=&v_{(-b-a,b,a)}-(q^{-1}-q)(v_{(-b,b+a,-a)}+v_{(b+a,-a,-b)})\\
&+(q^{-1}-q)^2(v_{(-a,b+a,-b)}+v_{(b,a,-b-a)})\\
&-(q^{-1}-q)(q^{-2}-1+q^2)v_{(a,b,-b-a)};\\
\overline{v_{(-b-a,a,b)}}=&v_{(-b-a,a,b)}-(q^{-1}-q)(v_{(b+a,-b,-a)}+v_{(-b-a,b,a)})\\
&+(q^{-1}-q)^2(v_{(b+a,-a,-b)}+v_{(-b,b+a,-a)})\\
&-(q^{-1}-q)(q^{-2}-1+q^2)(v_{(-a,b+a,-b)}+v_{(b,a,-b-a)})\\
&+(q^{-1}-q)^2(q^{-2}+q^2)(v_{(a,b,-b-a)});\\
\overline{v_{(b,-b-a,a)}}=&v_{(b,-b-a,a)}-(q^{-1}-q)(v_{(b+a,-b,-a)}+v_{(-b-a,b,a)})\\
&+(q^{-1}-q)^2(v_{(b+a,-a,-b)}+v_{(-b,b+a,-a)})\\
&-(q^{-1}-q)(q^{-2}-1+q^2)(v_{(-a,b+a,-b)}+v_{(b,a,-b-a)})\\
&+(q^{-1}-q)^2(q^{-2}+q^2)(v_{(a,b,-b-a)});\\
\overline{v_{(-b,-a,b+a)}}=&v_{(-b,-a,b+a)}-(q^{-1}-q)(v_{(-b-a,a,b)}+v_{(b,-b-a,a)})\\
&+(q^{-1}-q)^2(v_{(-b-a,b,a)}+v_{(b+a,-b,-a)})\\
&-(q^{-1}-q)(q^{-2}-1+q^2)(v_{(-b,b+a,-a)}+v_{(b+a,-a,-b)})\\
&+(q^{-1}-q)^2(q^{-2}+q^2)(v_{(-a,b+a,-b)}+v_{(b,a,-b-a)})\\
&-(q^{-1}-q)(q^{-4}-q^{-2}+1-q^2+q^4)v_{(a,b,-b-a)};\\
\overline{v_{(a,-b-a,b)}}=&v_{(a,-b-a,b)}-(q^{-1}-q)(v_{(-b-a,a,b)}+v_{(b,-b-a,a)})\\
&+(q^{-1}-q)^2(v_{(-b-a,b,a)}+v_{(b+a,-b,-a)})\\
&-(q^{-1}-q)(q^{-2}-1+q^2)(v_{(-b,b+a,-a)}+v_{(b+a,-a,-b)})\\
&+(q^{-1}-q)^2(q^{-2}+q^2)(v_{(-a,b+a,-b)}+v_{(b,a,-b-a)})\\
&-(q^{-1}-q)(q^{-4}-q^{-2}+1-q^2+q^4)v_{(a,b,-b-a)};\\
\overline{v_{(-a,-b,b+a)}}=&v_{(-a,-b,b+a)}-(q^{-1}-q)(v_{(-b,-a,b+a)}+v_{(a,-b-a,b)})
\\&+(q^{-1}-q)^2(v_{(-b-a,a,b)}+v_{(b,-b-a,a)})\\
&-(q^{-1}-q)(q^{-2}-1+q^2)(v_{(-b-a,b,a)}+v_{(b+a,-b,-a)})\\
&+(q^{-1}-q)^2(q^{-2}+q^2)(v_{(-b,b+a,-a)}+v_{(b+a,-a,-b)})\\
&-(q^{-1}-q)(q^{-4}-q^{-2}+1-q^2+q^4)(v_{(-a,b+a,-b)}+v_{(b,a,-b-a)})\\
&+(q^{-1}-q)^2(q^{-4}+1+q^4)v_{(a,b,-b-a)}.
\end{align*}

\subsection{Formulas for the actions of $e_a, f_a$ and $t$ on $\T_n$}
 \label{app:B}

Recall the sets ${}^iW$ of minimal length right coset representatives in $W$ with respect to the subgroups $\langle s_i\rangle$ $(i=1,2)$ are:
\begin{align*}
{}^1W &=\{1,s_2,s_2s_1,s_2s_1s_2,s_2s_1s_2 s_1, s_2s_1s_1s_1s_2\},\\
{}^2W &=\{1,s_1,s_1s_2,s_1s_2s_1,s_1s_2s_1s_2, s_1s_2s_1s_2s_1\}.
\end{align*}
The explicit formulas for the actions of $e_a, f_a$ and $t$ on the standard basis elements of $\T_n$ are given as follows:
\begin{align}
  \label{eq:e}
\begin{split}
e_{0}v_{(-1,-1,2)}&=v_{(-1,0,1)}+q^{-1}v_{(0,-1,1)};\\
e_{0}v_{(1,-2,1)}&=v_{(1,-1,0)}+v_{(0,-1,1)};\\
e_{0}v_{(-2,1,1)}&=v_{(-1,1,0)}+v_{(-1,0,1)};\\
e_{0}v_{(2,-1,-1)}&=v_{(1,0,-1)}+v_{(1,-1,0)};\\
e_{0}v_{(-1,2,-1)}&= v_{(0,1,-1)}+v_{(-1,1,0)};\\
e_{0}v_{(1,1,-2)}&=q v_{(0,1,-1)}+v_{(1,0,-1)};\\
e_0v_{(0,1,-1)} &=q e_0v_{(1,0,-1)}=q^{2}e_0v_{(-1,1,0)}
\\ &=q^{3}e_0v_{(1,-1,0)}=q^{4}e_0v_{(-1,0,1)}=q^5 e_0v_{(0,-1,1)}=v_{(0,0,0)};\\
e_0v_{(1,b,-b-1) \tau}&=qv_{(-1,b+1,-b) \tau}=v_{(0,b,-b) \tau}, \quad\forall 0\leq b\neq a,a+1; \tau\in {}^2W;\\
e_{a}v_{(a+1,a+1,-2a-2) \tau}&=q v_{(a,a+1,-2a-1) \tau}+v_{(a+1,a,-2a-1) \tau}, \quad\forall a\geq 1; \tau\in {}^1W;\\
e_{a}v_{(a,a+1,-2a-1) \tau}&=qe_{a}v_{(a+1,a,-2a-1) \tau}=v_{(a,a,-2a) \tau}, \quad \forall a\geq 1; \tau\in {}^1W;\\
e_av_{(a+1,b,-a-b-1) \tau}&=v_{(a,b,-a-b) \tau}, \quad\forall a>0, 0\leq b\neq a,a+1; \tau\in W;\\
e_a v_{(b,c,-b-c) \tau}&=0, \quad{\forall, 0\leq b,c\neq a+1; \tau\in W};
\end{split}
\end{align}

\begin{align}
  \label{eq:f}
\begin{split}
f_0v_{(0,0,0)}&=v_{(0,-1,1)}+q v_{(-1,0,1)}+q^2v_{(1,-1,0)}
\\
& \qquad\qquad\qquad +q^3v_{(-1,1,0)}+q^4v_{(1,0,-1)}+q^5 v_{(0,1,-1)};\\
f_0v_{(0,-1,1)}&=v_{(1,-2,1)}+q^{-1}v_{(-1,-1,2)};\\
f_0v_{(-1,0,1)}&=v_{(-2,1,1)}+v_{(-1,-1,2)};\\
f_0v_{(1,-1,0)}&=v_{(2,-1,-1)}+v_{(1,-2,1)};\\
f_0v_{(-1,1,0)}&=v_{(-1,2,-1)}+v_{(-2,1,1)};\\
f_0v_{(1,0,-1)}&=v_{(1,1,-2)}+v_{(2,-1,-1)};\\
f_0v_{(0,1,-1)}&=qv_{(1,1,-2)}+v_{(-1,2,-1)};\\
f_0v_{(0,b,-b) \tau}&=q v_{(1,b,-b-1) \tau}+v_{(-1,b+1,-b) \tau},\quad \forall b\geq 2; \tau\in {}^2W;\\
f_{a} v_{(a,a,-2a) \tau}&=q v_{(a,a+1,-2a-1) \tau}+v_{(a+1,a,-2a-1) \tau}, \quad\forall a\geq 1; \tau\in {}^1W;\\
f_{a} v_{(a,a+1,-2a-1) \tau}&=q f_{a} v_{(a+1,a,-2a-1) \tau}=v_{(a+1,a+1,-2a-2) \tau}, \quad\forall a\geq 1; \tau\in {}^1W;\\
f_a v_{(a,b,-a-b) \tau}&=v_{(a+1,b,-a-b-1) \tau}, \quad\forall a>0, 0\leq b\neq a,a+1; \tau\in W;\\
f_a v_{(b,c,-b-c) \tau}&=0, \quad{\forall, 0\leq b,c\neq a; \tau\in W};
\end{split}
\end{align}

\begin{align}
  \label{eq:t}
\begin{split}
tv_{(0,-1,1)}&=q^{-2}v_{(0,-1,1)}+q^{-1}v_{(-1,0,1)}+v_{(1,-1,0)};\\
tv_{(-1,0,1)}&=v_{(-1,0,1)}+v_{(-1,1,0)}+q^{-1}v_{(0,-1,1)};\\
tv_{(1,-1,0)}&=v_{(1,-1,0)}+v_{(0,-1,1)}+v_{(1,0,-1)};\\
tv_{(-1,1,0)}&=v_{(-1,1,0)}+v_{(-1,0,1)}+v_{(0,1,-1)};\\
tv_{(1,0,-1)}&=v_{(1,0,-1)}+v_{(1,-1,0)}+q v_{(0,1,-1)};\\
tv_{(0,1,-1)}&=q^{2}v_{(0,1,-1)}+q v_{(1,0,-1)}+v_{(-1,1,0)};\\
t v_{(a,a+1,-2a-1) \tau}&=qt v_{(a+1,a,-2a-1) \tau}
\\
&=q v_{(a,a+1,-2a-1) \tau}+v_{(a+1,a,-2a-1) \tau},
\;\forall a\geq1; \tau\in {}^1W;\\
tv_{(a,b,-a-b) \tau}&=v_{(a,b,-a-b)},\quad \forall 0<a+1<b; \tau\in W.\\
\end{split}
\end{align}

\subsection{Proof of Proposition \ref{generators}}
  \label{app:C}

The main results in Section~\ref{sec:Schur} on the $q$-Schur algebra still make sense for its $q=1$ specialization.
Recall the Schur algebra over $\mathbb{Z}$:
\[
\Sc(n)=\End_{W}(\T_n^1).
\]
Denote by $\Sc(n)_{\Q}=\Q\otimes_{\Z}\Sc(n)$ the schur algebra over $\Q$.
In particular, the results in Lemma~\ref{lem:coord}, Proposition~\ref{schurbasis}, Proposition~\ref{prop:basis}, Theorem~\ref{thm:double} as well as the definition of $e_a, f_a, t$ in \eqref{eq:eftq} carry over for $\Sc(n)$ and its module $\T_n^1$.

\begin{prop}
  \label{prop:generator1}
The Schur algebra $\Sc(n)_{\mathbb{Q}}$ is generated by $e_a, f_a (0\leq a<n)$ and $t$.
\end{prop}

\begin{proof}
Without loss of generality, we only need to prove the case of $n=2$ since all four types of the parabolic subgroups $ W_J$ (see \eqref{eq:PP}) appear in this case.
That is, to prove that  $\eta_{\mi,\mj}$ can be generated by $e_0, e_1, f_0, f_1$ and $t$ for any $(\mi,\mj)\in X_2^{(2)}$.
Indeed, we have
\begin{align}
\eta_{(0,0,0)(0,0,0)}&=\frac{1}{24}e_0^2 f_0^2; \notag \\
\eta_{(0,1,-1)(0,1,-1)}&= \frac{1}{2} e_1e_0f_0f_1, \notag \\
\eta_{(1,0,-1)(0,1,-1)}&=(t-1)\eta_{(0,1,-1)(0,1,-1)},\nonumber\\
\eta_{(1,-1,0)(0,1,-1)}&=(t^2-2t-1)\eta_{(0,1,-1)(0,1,-1)},\nonumber\\
\eta_{(0,-1,1)(0,1,-1)}&=\frac{1}{2}(t-1)(t^2-2t-2)\eta_{(0,1,-1)(0,1,-1)};\nonumber\\
\eta_{(0,1,-1)(0,0,0)}&=f_0\eta_{(0,0,0)(0,0,0)}, \notag  \\
\eta_{(0,0,0)(0,1,-1)}&=e_0\eta_{(0,1,-1)(0,1,-1)};\nonumber\\
\eta_{(1,1,-2)(1,1,-2)}&=\frac{1}{4}e_1^2f_1^2,\notag \\
\eta_{(-1,2,-1)(1,1,-2)}&=(f_0e_0-2)\eta_{(1,1,-2)(1,1,-2)},\nonumber\\
\eta_{(-2,1,1)(1,1,-2)}&=((f_0e_0-2)^2-2)\eta_{(1,1,-2)(1,1,-2)},\nonumber\\
\eta_{(-1,-1,2)(1,1,-2)}&=\frac{1}{2}(f_0\eta_{(0,-1,1)(0,1,-1)}e_0\eta_{(1,1,-2)(1,1,-2)}-\eta_{(-2,1,1)(1,1,-2)});\nonumber\\
\eta_{(1,1,-2)(0,0,0)}&=\frac{1}{2}f_0^2, \notag \\
\eta_{(0,0,0)(1,1,-2)}&=\frac{1}{2}e_0^2,\nonumber\\
\eta_{(1,1,-2)(0,1,-1)}&=f_0\eta_{(0,1,-1)(0,1,-1)},\nonumber\\
\eta_{(2,-1,-1)(0,1,-1)}&=f_0(e_0f_0-3)\eta_{(0,1,-1)(0,1,-1)},\nonumber\\
\eta_{(1,-2,1)(0,1,-1)}&=(f_0e_0-1)\eta_{(2,-1,-1)(0,1,-1)}-f_0(e_0f_0-2)\eta_{(0,1,-1)(0,1,-1)},\nonumber\\
\eta_{(0,1,-1)(1,1,-2)}&=e_0\eta_{(1,1,-2)(1,1,-2)},\nonumber\\
\eta_{(-1,1,0)(1,1,-2)}&=e_0(f_0e_0-3)\eta_{(1,1,-2)(1,1,-2)}\nonumber\\
\eta_{(-1,0,1)(1,1,-2)}&=(e_0f_0-1)\eta_{(-1,1,0)(1,1,-2)}-e_0(f_0e_0-2)\eta_{(1,1,-2)(1,1,-2)};\nonumber\\
\eta_{(0,2,-2)(0,2,-2)}&=\frac{1}{2}e_0f_0f_1e_1, \notag \\
\eta_{(2,0,-2)(0,2,-2)}&=f_1\eta_{(1,0,-1)(0,1,-1)}e_1,\nonumber\\
\eta_{(2,-2,0)(0,2,-2)}&=f_1\eta_{(1,-1,0)(0,1,-1)}e_1,\nonumber\\
\eta_{(0,-2,2)(0,2,-2)}&=f_1\eta_{(0,-1,1)(0,1,-1)}e_1;\nonumber\\
\eta_{(0,2,-2) w,(a,b,c)}&=f_1\eta_{(0,1,-1) w,(a,b,c)}, \notag \\
\eta_{(a,b,c) w,(0,2,-2)}&=\eta_{(a,b,c) w,(0,1,-1)}e_1,\nonumber\\
 \forall w\in W,& (a,b,c)=(0,0,0),(0,1,-1),(1,1,-2);\nonumber
\end{align}
\begin{align}
\label{eq:appendix}
\eta_{(1,2,-3)(1,2,-3)}&=1-\eta_{(0,0,0)(0,0,0)}-\eta_{(0,1,-1)(0,1,-1)}-\eta_{(1,1,-2)(1,1,-2)}\\
&\quad-\eta_{(0,2,-2)(0,2,-2)}-\eta_{(2,2,-4)(2,2,-4)}; \nonumber
\end{align}
\begin{align}
\eta_{(2,1,-3)(1,2,-3)}&=(t-1)\eta_{(1,2,-3)(1,2,-3)},\nonumber\\
\eta_{(-1,3,-2)(1,2,-3)}&=(f_0f_1e_1e_0-1)\eta_{(1,2,-3)(1,2,-3)},\nonumber\\
\eta_{(-2,3,-1)(1,2,-3)}&=(t-1)\eta_{(-1,3,-2)(1,2,-3)},\nonumber\\
\eta_{(3,-1,-2)(1,2,-3)}&=\eta_{(-1,3,-2)(1,2,-3)}\eta_{(2,1,-3)(1,2,-3)},\nonumber\\
\eta_{(3,-2,-1)(1,2,-3)}&=(t-1)\eta_{(3,-1,-2)(1,2,-3)},\nonumber\\
\eta_{(1,-3,2)(1,2,-3)}&=\eta_{(-2,3,-1)(1,2,-3)}\eta_{(3,-2,-1)(1,2,-3)},\nonumber\\
\eta_{(2,-3,1)(1,2,-3)}&=(t-1)\eta_{(1,-3,2)(1,2,-3)},\nonumber\\
\eta_{(-3,2,1)(1,2,-3)}&=\eta_{(-1,3,-2)(1,2,-3)}\eta_{(-2,3,-1)(1,2,-3)},\nonumber\\
\eta_{(-3,1,2)(1,2,-3)}&=(t-1)\eta_{(-3,2,1)(1,2,-3)},\nonumber\\
\eta_{(-1,-2,3)(1,2,-3)}&=\eta_{(1,-3,2)(1,2,-3)}\eta_{(-1,3,-2)(1,2,-3)},\nonumber\\
\eta_{(-2,-1,3)(1,2,-3)}&=(t-1)\eta_{(-1,-2,3)(1,2,-3)};\nonumber\\
\eta_{(0,0,0)(1,2,-3)}&=e_0e_1e_0, \notag \\
\eta_{(0,1,-1) w,(1,2,-3)}&=e_1e_0\eta_{(1,2,3) w,(1,2,-3)},\nonumber\\
\eta_{(1,1,-2) w,(1,2,-3)}&=e_1\eta_{(1,2,-3) w,(1,2,-3)},\nonumber\\
\eta_{(0,2,-2) w,(1,2,-3)}&=e_0\eta_{(1,2,-3) w,(1,2,-3)},\nonumber\\
\eta_{(1,2,-3),(0,0,-0)}&=f_0f_1f_0,\nonumber\\
\eta_{(1,2,-3) w,(0,1,-1)}&=\eta_{(1,2,-3) w,(1,2,-3)}f_0f_1,\nonumber\\
\eta_{(1,2,-3) w,(1,1,-2)}&=\eta_{(1,2,-3) w,(1,2,-3)}f_1,\nonumber\\
\eta_{(1,2,-3) w,(0,2,-2)}&=\eta_{(1,2,-3) w,(1,2,-3)}f_0,\quad \forall w\in W;\nonumber\\
\eta_{(2,2,-4)(2,2,-4)}&=\frac{1}{4}f_1^2e_1^2, \notag \\
\eta_{(-2,4,-2)(2,2,-4)}&=\frac{1}{4}f_1^2\eta_{(-1,2,-1)(1,1,-2)}e_1^2,\nonumber\\
\eta_{(-4,2,2)(2,2,-4)}&=\frac{1}{4}f_1^2\eta_{(-2,1,1)(1,1,-2)}e_1^2,\nonumber\\
\eta_{(-2,-2,4)(2,2,-4)}&=\frac{1}{4}f_1^2\eta_{(-1,-1,2)(1,1,-2)}e_1^2;\nonumber\\
\eta_{(2,2,-4) w,(a,b,c)}&=\frac{1}{2}f_1^2\eta_{(1,1,-2) w,(a,b,c)}, \notag \\
\eta_{(a,b,c) w,(2,2,-4)}&=\frac{1}{2}\eta_{(a,b,c) w,(1,1,-2)}e_1^2,\nonumber\\
 \forall w\in W,& (a,b,c)=(0,0,0),(0,1,-1),(0,2,-2),(1,1,-2),(1,2,-3).\nonumber
\end{align}
We remark that, for $n>2$, the equation for $\eta_{(1,2,-3)(1,2,-3)}$ in \eqref {eq:appendix} should be replaced by $\eta_{(1,2,-3)(1,2,-3)}=e_2e_1f_1f_2.$
(Note that $e_2$ and $f_2$ do not make sense for $n=2$.)
\end{proof}

We are ready to prove Proposition~ \ref{generators}, which is a $q$-analogue of Proposition~\ref{prop:generator1} above.

\begin{proof} [Proof of Proposition \ref{generators}]
In this proof, we shall write the $q=1$ version of $e_a, f_a, t$ as $e_a', f_a', t'$ (which are elements in the Schur algebra $\Sc(n)$). Denote by $N=\dim_\Q \Sc(n)_\Q =\dim_{\Q(q)} \Sc_q(n)_\Q$.
By Proposition~\ref{prop:generator1}, there exist polynomials $p_i(e_a',f_a',t')$, 
for $1\le i\le N$, which form a $\Q$-basis for $\Sc(n)_\Q$.

We claim that $\{p_i(e_a,f_a,t)~|~1\le i\le N\}$ form an $\Q(q)$-basis for $\Sc_q(n)_\Q$. Indeed, for dimension reason it suffices to show that they are linearly independent. Assume
\begin{equation}  \label{eq:indep}
\sum_{i} c_i(q) p_i(e_a,f_a,t)=0,
\end{equation}
for some $c_i(q) \in \Q(q)$, not all zero. By clearing common denominators we may assume all $c_i(q) \in \Q[q]$, and then by canceling the common factors we may further assume that the nonzero $c_i(q)$ are relatively prime. In particular, $c_i(1)\neq 0$ for some $i$. Taking the specialization at $q=1$ for \eqref{eq:indep}, we obtain that $\sum_{i} c_i(1) p_i(e_a',f_a',t')=0$, a contradiction. Hence the claim is proved, and Proposition \ref{generators} follows.
\end{proof}


\end{document}